\documentclass[11pt,reqno,a4paper]{amsart}
\usepackage{amsmath,amscd,amstext,amsthm,amsfonts,latexsym}
\usepackage[dvips]{graphicx}
\usepackage[utf8x]{inputenc}
\usepackage{xcolor}

\usepackage{amsfonts, amssymb, amsmath, amscd,   txfonts, graphicx, color }

\usepackage{   colordvi, amsthm, txfonts, enumerate}

\theoremstyle{plain}
\numberwithin{equation}{section}

\usepackage{colordvi, color}

\newtheorem{theorem}{Theorem}[section]
\newtheorem{proposition}[theorem]{Proposition}
\newtheorem{lemma}[theorem]{Lemma}
\newtheorem{corollary}[theorem]{Corollary}

\theoremstyle{definition}
\newtheorem{remark}[theorem]{Remark}
\newtheorem{example}[theorem]{Example}
\newtheorem{definition}[theorem]{Definition}

\newcommand{\field}[1]{\mathbb{#1}}
\newcommand{\RR}{\mathbb{R}}

\newcommand{\NN}{\field{N}}

\newcommand{\EE}{\mathbb{E}}

\newcommand{\MM}{\mathbb{M}}

\begin{document}

\title[Metric mean dimension and mean Hausdorff  dimension varying the metric]{Metric mean dimension and mean Hausdorff  dimension varying the metric}
\author{J. Muentes, A.J. Becker,  A.T. Baraviera, \'E. Scopel}
\date{\today}

\address{Jeovanny de Jesus Muentes Acevedo, Facultad de Ciencias B\'asicas,  Universidad Tecnol\'ogica de  Bol\'ivar, Cartagena de Indias - Colombia}
\email{jmuentes@utb.edu.co}

\address{Alex Jenaro Becker, Universidade Federal de Santa Maria, Santa Maria, RS - Brasil}
\email{alex.becker@ufsm.br}

\address{Alexandre Tavares Baraviera, Universidade Federal do
Rio Grande do Sul, Porto Alegre, RS - Brasil}
\email{baravi@mat.ufrgs.br}

\address{Érick Scopel, Instituto Federal de Educação, Ciência e Tecnologia do Rio Grande do Sul, Caxias do Sul, RS - Brasil}
\email{erick.scopel@caxias.ifrs.edu.br}

\begin{abstract} Let $f:\MM\rightarrow \MM$ be a continuous map on a compact metric space $\MM$ equipped with a fixed metric $d$,  and let  $\tau$ be the topology on $\MM$ induced by $d$.  
First, we will establish some fundamental properties of the mean Hausdorff dimension. Furthermore, it is important to note that the metric mean dimension and mean Hausdorff dimension depend on the metric chosen for  $\MM$.  In this work, we will prove that, for a fixed dynamical system $f:\MM\rightarrow \MM$, the functions $\text{mdim}_{\text{M}}(\MM, f):\MM(\tau)\rightarrow \mathbb{R}\cup \{\infty\}$ and $\text{mdim}_{\text{H}}(\MM, f):\MM(\tau)\rightarrow \mathbb{R}\cup \{\infty\}$ are not   continuous. Here, $ \text{mdim}_{\text{M}}(\MM, f)(\rho)= \text{mdim}_{\text{M}}(\MM,\rho, f)$ and $ \text{mdim}_{\text{H}}(\MM, f)(\rho)= \text{mdim}_{\text{H}}(\MM,\rho, f)$ represent, respectively,  the metric mean dimension and the mean Hausdorff dimension  of $f$ with respect to $\rho\in \MM(\tau)$ and   $\MM(\tau)$ is the set consisting of all   equivalent metrics to $d$  on $\MM$. Furthermore, we will present examples of  certain classes of metrics for which  the metric mean dimension   is a  continuous function.  
\end{abstract}

\keywords{mean topological dimension, metric mean dimension, mean Hausdorff dimension, topological entropy, box dimension, Hausdorff dimension}

\subjclass[2020]{37B40, 	 	37B02}

\date{\today}
\maketitle

\section{Introduction}

The mean topological dimension of a dynamical system $(\MM, f)$, denoted by $\text{mdim}(\MM,f)$,   where $\MM$ is a  compact topological space and $f$ is a continuous map, is an invariant under topological conjugacy.   This concept was introduced by Gromov in 1999 (\cite{Gromov}). It serves as an essential tool for understanding systems with infinite topological entropy.  In 2000, Lindenstrauss and Weiss (\cite{lind}) demonstrated that the left-shift map defined on $([0,1]^{n})^\mathbb{Z}$ has a mean topological dimension equal to $n$, where $n$ is a positive integer.  We define the mean topological dimension in Section \ref{section2}.

\medskip

The concept of mean topological dimension is closely related to problems involving the embedding of minimal dynamical systems. The works \cite{lind}, \cite{elon}, \cite{gutman}, and \cite{GTM} demonstrate that any minimal system with a mean topological dimension less than $\frac{n}{2}$ can be embedded into the shift map on $([0,1]^n)^\mathbb{Z}$.   It is worth noting that the value $\frac{n}{2}$ is optimal in this context. In \cite{Dou}, the author constructed minimal subshifts on a countable infinite amenable group with arbitrarily mean topological dimension.     It is also worth mentioning that calculating the mean topological dimension is a challenging task. Consequently, it becomes crucial to obtain upper bounds for the mean topological dimension of a dynamical system. 

\medskip

The metric mean dimension for dynamical systems defined on  compact metric spaces, introduced by Lindenstrauss and Weiss in 2000 (\cite{lind}), offers upper bounds for the mean topological dimension. Since its introduction, the notion of metric mean dimension has been extensively studied, as we can see in the works  \cite{Backes}, \cite{Carvalho2}, \cite{Cheng}, \cite{{Lacerda}}, \cite{Ma}, \cite{Yang}, among other works. 

\medskip 

In 2019, Lindenstrauss and Tsukamoto (\cite{LT2019}) introduced a new tool that provides a better upper bound for the mean topological dimension: the mean Hausdorff dimension. However, it is important to note that both the metric mean dimension and mean Hausdorff dimension are not invariant under topological conjugacy; they depend on the chosen metric for the  space.

\medskip 
In summary, the metric mean dimension and mean Hausdorff dimension  depend on three variables: the dynamics represented by $f$, the space denoted as $\MM$, and the metric $d$ employed on $\MM$.  We denote by $\text{mdim}_{\text{M}}(\MM, d, f)$ and $\text{mdim}_{\text{H}}(\MM, d, f)$ the metric mean dimension and the mean Hausdorff dimension of $f$, respectively.

\medskip

Several works explore the metric mean dimension concerning the dynamics and the invariant space in which these dynamics operate. For instance, in \cite{Carvalho}, the authors establish that, for $C^0$-generic homeomorphisms acting on a compact, smooth, boundaryless manifold $\MM$ with  dimension greater than one, the upper metric mean dimension concerning the smooth metric coincides with the dimension of the manifold.   Furthermore, in \cite{Muentes2} it is proved the  set of all homeomorphisms on $\MM$ with metric mean dimension equal to a fixed $\alpha \in [0,\text{dim}(\MM)]$ is dense in $\text{Hom}(\MM)$, where $\text{dim}(\MM)$ is the topological dimension of $\MM$. These results are similarly demonstrated in \cite{Muentes3} for the case of the mean Hausdorff dimension. 
Moreover, in \cite{Muentes} it is proved that  if $\text{dim}(\MM)\geq 2$, the mapping
\begin{equation*}
\begin{aligned}
\text{mdim}_{\text{M}}(\MM,d,\cdot ) \colon \text{Hom}(\MM) &\to \mathbb{R}\\
f &\mapsto \text{mdim}_{\text{M}}(\MM,d,f)
\end{aligned}
\end{equation*}
is not continuous anywhere.  

\medskip

   The dependence of the metric mean dimension on the metric has been explored in various works. For instance,   in \cite{lind}  it is proven that for any metric $d$ on $\MM$, we have $$ 
\text{mdim}(\MM,f)\leq  \overline{\text{mdim}}_{\text{M}}(\MM,d,f) .$$
Furthermore, it is conjectured that for any dynamical system $(\MM, f)$, there exists a metric $d$ on $\MM$ such that $$ 
\text{mdim}(\MM,f)=  \overline{\text{mdim}}_{\text{M}}(\MM,d,f) .$$ 
This conjecture has been verified for specific cases of dynamical systems (see \cite{LT2019}, Theorem 3.12). 
In \cite{LT2019}, the authors  present an example of a left shift $(A^{\mathbb{Z}},\sigma)$ and two metrics $d$ and $d^{\prime}$ on $A^{\mathbb{Z}}$ such that  $${\text{mdim}}_{\text{M}}(A^{\mathbb{Z}},d,\sigma)= \frac{1}{2}=\text{dim}_{\text{B}}(A)\quad\text{and}\quad {\text{mdim}}_{\text{M}}(A^{\mathbb{Z}},d^{\prime},\sigma)=0, $$ where $\text{dim}_{\text{B}}(A)$ denotes the box dimension of A (for the definition of box dimension, see \cite{Falconer}, Section 3.1).  In Example \ref{example11}, we will provide an example of a fixed dynamical system $f:[0,1]\rightarrow [0,1]$ such that for any fixed $a \in [0,1]$  there exists an explicit  metric   $d_{a}$ on $[0,1]$ such that $\text{mdim}_{\text{M}}([0,1], d_{a}, f) =\text{mdim}_{\text{H}}([0,1], d_{a}, f)=a$  (see Remark \ref{wfwfaffa}).

\medskip 

In \cite{Carvalho}, Corollary D states that there exist a dense subset of metrics $\mathcal{D}$ on $[0,1]$ and a   generic subset $\mathcal{G}$ of
$C^{0}([0, 1])$  such that
$$\text{mdim}_{\text{M}} ([0, 1] , \rho, f) = 1 \text{ for all }f \in \mathcal{G}, \text{ for all }\rho \in \mathcal{D}.$$

Next, in \cite{ShinodaM}, Theorem 1.1   states that if $A$ is a finite set, then
$$\text{mdim}_{\text{M}}(\mathcal{X},d_{\alpha}, \sigma_{1})=\frac{2h_{\text{top}}(\mathcal{X}, \sigma_{1},\sigma_{2})}{\log \alpha}  ,$$  where $\sigma_{1}((x_{m,n})_{m,n\in\mathbb{Z}})=((x_{m+1,n})_{m,n\in\mathbb{Z}})$ and $\sigma_{2}((x_{m,n})_{m,n\in\mathbb{Z}})=((x_{m,n+1})_{m,n\in\mathbb{Z}})$ are defined in $A^{\mathbb{Z}^{2}}$, $  \mathcal{X} $ is a closed subset of $A^{\mathbb{Z}^{2}}$ invariant under both $\sigma_{1}$ and $\sigma_{2}$ and $$ d_{\alpha}(x,y)=\alpha^{-\min\{| u|_{\infty}:x_{u}\neq y_{u}\}},$$
where $| u|_{\infty}= \max(|m|,|n|)$ for $u=(m,n)\in\mathbb{Z}^{2}$ and $\alpha>1$.  In Examples    \ref{gshfkfwee} and  \ref{gshfkf}, we will consider a similar metric $\textbf{d}_{\alpha}$ on the Cantor set $\textit{\textbf{C}}$ and  calculate the metric mean dimension of some particular maps on $(\textit{\textbf{C}}, \textbf{d}_{\alpha})$. 

\medskip 

From Examples \ref{metrica} and \ref{efjfk}, we can conclude that, for any $b\in [n,\infty)$, there exists a metric $d_{b}$ on $([0,1]^{n})^{\mathbb{Z}}$ such that $$\text{mdim}_{\text{M}}(([0,1]^{n})^{\mathbb{Z}},d_{b},\sigma)=\text{mdim}_{\text{H}}(([0,1]^{n})^{\mathbb{Z}},d_{b},\sigma)=b$$ (see \eqref{fwfwfswfw} and \eqref{gegwvxfmhm}).  

\medskip

The purpose of this work is to explore the continuity of the metric mean dimension on the metric $d$ on $\MM$. We will prove that, in general, the functions  $d\mapsto \text{mdim}_{\text{M}}(\MM, d,f) $ and $d\mapsto \text{mdim}_{\text{H}}(\MM, d,f) $ are not continuous anywhere. On the other hand, we will present examples of certain classes of metrics for which $d\mapsto \text{mdim}_{\text{M}}(\MM, d,f) $ and $d\mapsto \text{mdim}_{\text{H}}(\MM, d,f) $ are continuous functions. 
 
  \medskip

 The paper is organized as follows: in the next section, we will introduce the concepts of mean topological dimension, metric mean dimension and mean Hausdorff dimension. Furthermore, we will present some alternative formulas to calculate  the  Hausdorff dimension of any   compact metric  space, which are more aligned with the definition of mean Hausdorff dimension for dynamical systems (see Lemmas \ref{mg4g4l} and \ref{fege6}).  
 
In Section \ref{section4}, we will establish several properties of the mean Hausdorff dimension, inspired by properties already known for the metric mean dimension and based on the foundational concepts of the Hausdorff dimension.  For instance, it is well known that, given two metric spaces $(\MM,d)$ and $(\EE,d^{\prime})$, we have that
$$\text{dim}_{\text{H}}(\MM\times\EE )\geq \text{dim}_{\text{H}}(\MM)+\text{dim}_{\text{H}}(\EE)  $$ (see \cite{Falconer}, Chapter 7).  
In Proposition \ref{prdoctv}, we show that $$\underline{\text{mdim}}_{\text{H}}(\MM \times \EE,  d \times d^{\prime}, f \times g) \geq \underline{\text{mdim}}_{\text{H}}(\MM, d, f) + \underline{\text{mdim}}_{\text{H}}(\EE, d' , g),$$ for any two maps $f:(\MM,d)\rightarrow (\MM,d)$ and $g:(\EE,d^{\prime})\rightarrow (\EE,d^{\prime})$. Furthermore, in Theorem \ref{theo1}, we prove that, for $\mathbb{K}=\mathbb{Z} $ or $\mathbb{N}$,   
$$ \text{dim}_{\text{H}}(\MM,d)\leq \underline{\text{mdim}}_{\text{H}}(\MM^{\mathbb{K}} ,\text{\textbf{d}},\sigma)  ,$$
where $\sigma:\MM^{\mathbb{K}} \rightarrow \MM^{\mathbb{K}}  $ is the left shift map and $\textbf{d}$ is  a  specific metric on $\MM^{\mathbb{K}}$ obtained from the metric $d$  on $\MM$ (see \eqref{metrionmk}). In order to obtain this result, we use  Lemma \ref{czzzvrv}, in which  we present an alternative formula to calculate $ {\text{mdim}}_{\text{H}}(\MM^{\mathbb{K}} ,\text{\textbf{d}},\sigma)  .$

In Section \ref{someexamples}, we will calculate the metric mean dimension of several continuous maps $f:\MM\rightarrow \MM$ changing the metric on  $\MM$, when $\MM$ is the interval $[0,1]$ or the Cantor set.

In Section \ref{section5}, we will prove that both  the metric mean dimension and the mean Hausdorff dimension are not continuous with respect the metric. 

 In Section \ref{section7}, we will consider certain classes of metrics and explore how the metric mean dimension behaves when these metrics vary within these classes. More specifically, we will generate metrics using composition of subadditive continuous maps with a fixed metric on $\MM$.

We conclude this work by presenting some illustrative examples in Section \ref{sectionfinal}.

\section{Mean dimension, metric mean dimension and mean Hausdorff dimension}\label{section2}

Throughout this work, we will fix a metrizable compact space   $\MM$ and we will fix a metric $d$ on $\MM$, compatible with the topology on $\MM$. In this section we will present the notions of mean topological dimension, metric mean dimension and   mean Hausdorff dimension, introduced in \cite{lind} and \cite{LT2019}, respectively. 

\medskip
We briefly present the definition of mean topological dimension.   Let $\alpha=\{A_i\}_i$ be an open cover of $\MM$ and define $\text{ord}(\alpha)=\displaystyle\sup_{x\in X}\displaystyle\sum_{A_i\in\alpha} 1_{A_i}(x) -1$. A {\it refinement} of $\alpha$ is an open cover $\beta=\{B_j\}_j$ such that for any $B_j\in \beta$, there exists $A_i\in \alpha$, such that $B_j\subset A_i$. When $\beta$ is a refinement of $\alpha$, we write $\beta \succ \alpha$. Set $D(\alpha)=\displaystyle\min_{\beta\succ\alpha} \text{ord}(\beta)$, where $\alpha$ runs over all finite open covers of $\MM$  refining $\alpha$. The {\it topological dimension} of $\MM$ is $$\dim(\MM)=\sup \{D(\alpha):\alpha\text{ is a cover of }\MM\}.$$ 

Consider any continuous function $f:\MM\to \MM$, the {\it mean topological dimension} is defined as follow
$$\text{mdim} (\MM,f)=\displaystyle\sup_\alpha\lim_{n\to \infty}\frac{D(\alpha \vee f^{-1}(\alpha)\vee \cdots \vee f^{-n+1}(\alpha))}{n},$$ where $\alpha $   runs over all finite open covers of $\MM$.
The sequence $\alpha \vee f^{-1}(\alpha)\vee \cdots \vee f^{-n+1}(\alpha)$ is subadditive for $n\geq 1$, and the above limit exists.

\medskip

For a continuous map $f:\MM \to \MM$ and any non-negative integer
$n$, set
$$
d_n(x,y)=\max \left\{d(x,y),d(f(x),f(y)),\dots,d(f^{n-1}(x),f^{n-1}(y))\right\}\quad\text{for any }x,y\in\MM.
$$   
We say that $A\subset \MM$ is {an} $(n,f,\varepsilon)$-\textit{separated subset}
if $d_n(x,y)>\varepsilon$, for any two  distinct points  $x,y\in A$. We denote by $\text{sep}(n,f,\varepsilon)$ the maximal cardinality of any $(n,f,\varepsilon)$-separated
subset of $\MM$.   We say that $E\subset \MM$ is {an} $(n,f,\varepsilon)$-\textit{spanning} set for $\MM$ if
for any $x\in \MM$ there exists $y\in E$ such  that $d_n(x,y)<\varepsilon$. Let $\text{span}(n,f,\varepsilon)$ be the minimum cardinality
of any $(n,f,\varepsilon)$-{spanning} subset of $\MM$.   Given an open cover $\alpha$ of $\MM$, we say that $\alpha$ is {an}
$(n, f,\varepsilon)$-\textit{cover} of $\MM$ if the $d_n$-diameter of any element of $\alpha$ is less than
 $\varepsilon$. Let  $\text{cov}(n, f,\varepsilon)$ be the minimum number of elements in any $(n,f,\varepsilon)$-cover of $\MM$.    Set \begin{itemize} \item $\text{sep}(f,\varepsilon)=\underset{n\to\infty}\limsup \frac{1}{n}\log \text{sep}(n,f,\varepsilon)$;
\item $\text{span}(f,\varepsilon)=\underset{n\to\infty}\limsup \frac{1}{n}\log \text{span}(n,f,\varepsilon)$;
\item  $\text{cov}(f,\varepsilon)=\underset{n\to\infty}\limsup \frac{1}{n}\log \text{cov}(n,f,\varepsilon)$.\end{itemize}

 \begin{definition}
  We define the \emph{lower  metric mean dimension} of $(\MM,d,f)$  and the \emph{upper metric mean dimension} of $(\MM,d,f)$ by
  \begin{equation*}\label{metric-mean}
 \underline{\text{mdim}}_{\text{M}}(\MM,d,f)=\liminf_{\varepsilon\to0} \frac{\text{sep}(f,\varepsilon)}{|\log \varepsilon|}=\liminf_{\varepsilon\to0} \frac{\text{span}(f,\varepsilon)}{|\log \varepsilon|}=\liminf_{\varepsilon\to0} \frac{\text{cov}(f,\varepsilon)}{|\log \varepsilon|}\end{equation*}
and \begin{equation*}
\overline{\text{mdim}}_{\text{M}}(\MM,d,f)=\limsup_{\varepsilon\to0} \frac{\text{sep}(f,\varepsilon)}{|\log \varepsilon|}=\limsup_{\varepsilon\to0} \frac{\text{span}(f,\varepsilon)}{|\log \varepsilon|}=\limsup_{\varepsilon\to0} \frac{\text{cov}(f,\varepsilon)}{|\log \varepsilon|},
\end{equation*}
respectively (see \cite{lind}, Section 4). \end{definition}

Now, we present the definition of the Hausdorff dimension given in   \cite{LT2019}:
for $s\geq 0$ and $\varepsilon >0$, set
\begin{equation}\label{vareps}\text{H}_{\varepsilon}^{s} (\MM,d)=\inf\left\{ \Sigma_{n=1}^{\infty}(\text{diam} E_{n})^{s}: \MM=\cup_{n=1}^{\infty} E_{n} \text{ with } \text{diam} E_{n}<\varepsilon\text{ for all }n\geq 1\right\}. \end{equation} By  convention we consider $0^{0}=1$ and $\text{diam}(\emptyset)^{s}=0$. Let $\Theta >0$. Take $$ \text{dim}_{\text{H}}(\MM,d,\varepsilon,\Theta)=\sup\{s \geq 0:  \text{H}_{\varepsilon}^{s} (\MM,d) \geq \Theta\}.  $$
The \textit{Hausdorff dimension} of $(\MM,d)$, presented in \cite{LT2019},
is given by   $$ \text{dim}_{\text{H}}(\MM,d):= \lim_{\varepsilon\rightarrow 0}\text{dim}_{\text{H}}(\MM,d,\varepsilon,1).  $$
By simplicity in the notation, if $\Theta = 1$, we will set $$ \text{dim}_{\text{H}}(\MM,d,\varepsilon ):= \text{dim}_{\text{H}}(\MM,d,\varepsilon,1).$$
 
The usual definition of the Hausdorff dimension in the literature it is as follows:  let
$$\text{H}^s(\MM,d)=\displaystyle\lim_{\varepsilon\to 0}\text{H}^s_\varepsilon(\MM,d).$$ The \textit{Hausdorff dimension} of  $(\MM,d)$, denoted by $\dim_{\text{H}}^\ast(\MM,d)$, is given by
$$\dim_{\text{H}}^\ast(\MM,d)=\sup\{s\geq 0: \text{H}^s(\MM,d)>0\}=\sup\{s\geq 0: \text{H}^s(\MM,d)=\infty\}.$$

\begin{lemma}\label{mg4g4l} Fix any $\Theta >0$. 
We have that $\dim_{\emph{H}}(\MM,d)=\dim_{\emph{H}}^\ast(\MM,d)$ and furthermore $$ \dim_{\emph{H}}^{\Theta}(\MM,d):=\lim_{\varepsilon\rightarrow 0}\emph{dim}_{\emph{H}}(\MM,d,\varepsilon,\Theta)=\dim_{\emph{H}}(\MM,d).$$\end{lemma}
\begin{proof} First, notice that if   $\varepsilon >0$  in \eqref{vareps} decreases, the class of permissible  covers of $\MM$, with diameter less than $\varepsilon$, decreases. Therefore, for any $s\geq 0$, $\text{H}^s_\varepsilon(\MM,d)$ increases as $\varepsilon $ decreases. Hence, $$\text{H}^s_\varepsilon(\MM,d)\leq \text{H}^s(\MM,d)\quad \text{ for any }\varepsilon>0.$$ Thus, if $s\geq 0$ is such that $\text{H}^s_\varepsilon(\MM,d)\geq \Theta$, we have that $\text{H}^s(\MM,d)>0$. Consequently,
\begin{eqnarray*}
\dim_{\text{H}}(\MM,d,\varepsilon,\Theta)=\sup\{s\geq 0: \text{H}^s_\varepsilon(\MM,d)\geq \Theta\}&\leq &\sup\{s\geq 0: \text{H}^s(\MM,d)>0\}\\
&=&\dim_{\text{H}}^\ast(\MM,d).
\end{eqnarray*}
Taking the limit as $\varepsilon\to 0$, we obtain that \begin{equation}\label{fefezzzx}\dim_{\text{H}}^{\Theta}(\MM,d)\leq \dim_{\text{H}}^\ast(\MM,d).\end{equation}

Next, notice that, if $\dim_{\text{H}}^\ast(\MM,d)=0$, then $\dim_{\text{H}}^{\Theta}(\MM,d)=0$.  Suppose that $\dim_{\text{H}}^\ast(\MM,d)>0$. From the definition, for each $\delta>0$  there exists $s_\delta> 0$ such that
$$\dim_{\text{H}}^\ast(\MM,d)-\delta <s_\delta\leq \dim_{\text{H}}^\ast(\MM,d) \quad \text{and} \quad \text{H}^{s_\delta}(\MM,d)=\infty.$$
Thus, there exists $\varepsilon_0$ such that $\text{H}_\varepsilon^{s_\delta}(\MM,d)\geq \Theta$, for every $0<\varepsilon<\varepsilon_0$. Hence,
$$\dim_{\text{H}}(\MM,d,\varepsilon,\Theta)\geq s_\delta > \dim_{\text{H}}^\ast(\MM,d)-\delta.$$
Taking the limits as $\varepsilon\to 0$ and $\delta\to 0$, we conclude that \begin{equation}\label{fefezzczzx}\dim_{\text{H}}^{\Theta}(\MM,d)\geq \dim_{\text{H}}^\ast(\MM,d).\end{equation} From   \eqref{fefezzzx} and \eqref{fefezzczzx} we have that $\dim_{\text{H}}^{\Theta}(\MM,d)$ is independent of $\Theta>0$ and furthermore $$ \dim_{\text{H}}(\MM,d)= \dim_{\text{H}}^{\Theta}(\MM,d)=\dim_{\text{H}}^{\ast}(\MM,d),$$ as we want to prove.
\end{proof}

 \begin{lemma}\label{fege6} Suppose that $(\MM,d)$ is a compact space. 
 For    $s\geq 0$ and $\varepsilon >0$, set $$  \emph{B}_{\varepsilon}^{s} (\MM,d)=\inf\left\{ \Sigma_{n=1}^{m}(\emph{diam} (B_{n}))^{s}:\{ B_{n}\}_{n=1}^{m} \text{  is a   cover of }\MM  \text{ by open balls  with }\emph{diam} (B_n) \leq  \varepsilon\right\}.$$   Setting  
 $$   {\emph{dim}}^{\star}_{\emph{H}}(\MM,d,\varepsilon)=\sup\{s \geq 0:  \emph{B}_{\varepsilon}^{s} (\MM,d) \geq 1\},  $$
 we have that $$ \dim_{\emph{H}} (\MM,d)=\lim_{\varepsilon\rightarrow 0}\emph{dim}^{\star}_{\emph{H}}(\MM,d,\varepsilon).$$
 \end{lemma}
\begin{proof} We can prove that \begin{equation} \label{nbsff} \text{H}_{\varepsilon}^{s} (\MM,d)\leq  \text{B}_{\varepsilon}^{s} (\MM,d)\leq  2^{s}\text{H}_{{\varepsilon}/{2}}^{s} (\MM,d) \end{equation} (see \cite{Falconer}, Section 2.4). It follows from the first inequality in \eqref{nbsff} that  
 \begin{equation} \label{fefwfwfzx}  {\text{dim}}_{\text{H}}(\MM,d,\varepsilon)\leq {\text{dim}}^{\star}_{\text{H}}(\MM,d,\varepsilon).  \end{equation}  
  Next, if $t$ is such that $1\leq \text{B}_{\varepsilon}^{t} (\MM,d)$, then by \eqref{nbsff}   we have $\frac{1}{2^{t}}\leq \text{H}_{\varepsilon/2}^{t} (\MM,d)$. Therefore, 
  \begin{equation} \label{xfddefwfwfzx}{\text{dim}}^{\star}_{\text{H}}(\MM,d,\varepsilon) \leq \dim_{\text{H}}(\MM,d,\varepsilon/2,{1}/{2^{t}})  .  \end{equation}

From  \eqref{fefwfwfzx},  \eqref{xfddefwfwfzx} and Lemma \ref{mg4g4l}, we have   that  $$  \text{dim}_{\text{H}}(\MM,d)= \lim_{\varepsilon\rightarrow 0}{\text{dim}}^{\star}_{\text{H}}(\MM,d,\varepsilon) ,$$   as we want to prove.
\end{proof}

\begin{definition}
The \textit{upper mean Hausdorff dimension} and \textit{lower mean Hausdorff dimension} of $(\MM,d,f) $ are defined respectively as
$$ \overline{\text{mdim}}_{\text{H}}(\MM,d,f)=\lim_{\varepsilon\rightarrow 0} \left(\limsup_{n\rightarrow \infty}\frac{1}{n} \text{dim}_{\text{H}}(\MM,d_{n},\varepsilon)\right)=\lim_{\varepsilon\rightarrow 0} \left(\limsup_{n\rightarrow \infty}\frac{1}{n} \text{dim}^{\star}_{\text{H}}(\MM,d_{n},\varepsilon)\right), $$ $$\underline{\text{mdim}}_{\text{H}}(\MM,d,f)=\lim_{\varepsilon\rightarrow 0} \left(\liminf_{n\rightarrow \infty}\frac{1}{n} \text{dim}_{\text{H}}(\MM,d_{n},\varepsilon)\right)=\lim_{\varepsilon\rightarrow 0} \left(\liminf_{n\rightarrow \infty}\frac{1}{n} \text{dim}^{\star}_{\text{H}}(\MM,d_{n},\varepsilon)\right) $$ (see \cite{LT2019}, Section 3).
\end{definition}

\begin{remark}\label{obsscc} Denote by $\text{mdim}(\MM,f)$ the mean dimension of a  continuous map $f:\MM\rightarrow \MM$ (see \cite{lind}). The inequalities \begin{align*}
\text{mdim}(\MM,f)&\leq  \underline{\text{mdim}}_{\text{H}}(\MM,d,f) \leq \overline{\text{mdim}}_{\text{H}}(\MM,d,f)  \leq  \underline{\text{mdim}}_{\text{M}}(\MM,d,f)\leq \overline{\text{mdim}}_{\text{M}}(\MM,d,f)  \end{align*} always hold (see \cite{LT2019}). \end{remark}

Recently, in \cite{Liu}, the authors introduce the  concepts of mean packing dimension and mean pseudo-packing dimension for dynamical systems.  They proved that  the mean Hausdorff dimension of a dynamical system is lower than its mean packing dimension and its mean pseudo-packing dimension. Hence, the mean Hausdorff dimension remains a more accurate approximation of the mean topological dimension. 


\section{Some fundamental properties of the  mean Hausdorff dimension}\label{section4}

 Let $f:\MM\rightarrow \MM$ be a continuous map, and let $A \subset \MM$ be a   non-empty closed subset that is invariant under $f$. It is straightforward to observe that: 
$$\overline{\text{mdim}}_{\text{H}}(A,d,f|_{A}) \leq \overline{\text{mdim}}_{\text{H}}(\MM,d,f)\quad\text{and}\quad \underline{\text{mdim}}_{\text{H}}(A,d,f|_{A}) \leq \underline{\text{mdim}}_{\text{H}}(\MM,d,f).$$

Next, it is well-known that for any   $p\in \mathbb N$, 
we have 
  $$
  {\text{mdim}_{\text{M}}}(\MM,d,f^{p})\leq p\,  {\text{mdim}_{\text{M}}}(\MM,d,f).
  $$  In \cite{Muentes}, Corollary 3.4 provides a formula for $ {\text{mdim}_{\text{M}}}(\MM,d,f^{p})$ for a certain class of continuous maps on the interval (see Remark \ref{observacr}).   For the mean Hausdorff dimension, similar relationships apply. 

\begin{proposition}\label{propo211} Let $f:\MM\rightarrow \MM$ be a continuous map. 
For any   $p\in \mathbb N$, we have
  $$
  \underline{\emph{mdim}}_{\emph{H}}(\MM,d,f^{p})\leq p\,  \underline{\emph{mdim}}_{\emph{H}}(\MM,d,f)\quad \text{and}\quad 
  \overline{\emph{mdim}}_{\emph{H}}(\MM,d,f^{p})\leq p\,  \overline{\emph{mdim}}_{\emph{H}}(\MM,d,f).
  $$ 
\end{proposition} \begin{proof}
For any positive integer $m$, we know that
$$\max_{0\leq j<m}d(f^{jp}(x),f^{jp}(y))\leq \max_{0\leq j<mp}d(f^{j}(x),f^{j}(y)).$$ 
Hence, for each $s\geq 0$ and $ \varepsilon >0$, we have \begin{align*} 
\text{H}_{\varepsilon}^{s} (\MM,d_{m},f^{p}) &=\inf\left\{ \Sigma_{n=1}^{\infty}(\underset{d_{m,f^{p}}}{\text{diam}} E_{n})^{s}: \MM=\cup_{n=1}^{\infty} E_{n} \text{ with } \underset{d_{m,f^{p}}}{\text{diam}} E_{n}<\varepsilon\text{ for all }n\geq 1\right\}\\
&\leq \inf\left\{ \Sigma_{n=1}^{\infty}(\underset{d_{mp,f}}{\text{diam}} E_{n})^{s}: \MM=\cup_{n=1}^{\infty} E_{n} \text{ with } \underset{d_{mp,f}}{\text{diam}} E_{n}<\varepsilon\text{ for all }n\geq 1\right\}\\
&=\text{H}_{\varepsilon}^{s}(\MM,d_{mp},f),\end{align*}
where $ \underset{d_{m,f}}{\text{diam}}$ represents the diameter with respect to the dynamic metric $d_{m}$ associated to $f$.   Therefore, $$\text{dim}_{\text{H}}(\MM,d_{m},\varepsilon, f^{p} )\leq  \text{dim}_{\text{H}}(\MM,d_{mp},\varepsilon, f  ) $$ and hence
$$
\limsup_{m\to \infty}\frac{1}{m} \text{dim}_{\text{H}}(\MM,d_{m},\varepsilon, f^{p} )
                                    \leq p\limsup_{m\to \infty}\frac{1}{mp}   \text{dim}_{\text{H}}(\MM,d_{mp},\varepsilon, f  ). $$ 
This fact proves the proposition. 
\end{proof}

 Next, consider two continuous maps $f \colon \MM \to \MM$ and $g \colon \EE \to \EE$, where $(\MM,d)$ and $(\EE,d')$ are compact metric spaces. We will endow the product space $\MM \times \EE$ with the metric 
 \begin{equation}\label{bnm}(d\times d^{\prime})((x_{1},y_{1}),(x_{2},y_{2}))=\max\{d(x_{1},x_{2}),d^{\prime}(y_{1},y_{2})\}, \text{ for } x_1,x_2 \in \MM \text{ and } y_1,y_2 \in \EE.\end{equation} 
 This metric is uniformly equivalent  to (see Remark \ref{fefefeesee}) the both metrics
\begin{equation*}\label{bnms}(d\times d^{\prime})^{\ast}((x_{1},y_{1}),(x_{2},y_{2}))=d(x_{1},x_{2}) + d^{\prime}(y_{1},y_{2}), \text{ for } x_1,x_2 \in \MM \text{ and } y_1,y_2 \in \EE.\end{equation*} 
\begin{equation*}\label{zqwbnm}\overline{(d\times d^{\prime})}((x_{1},y_{1}),(x_{2},y_{2}))=\sqrt{d(x_{1},x_{2})^{2} + d^{\prime}(y_{1},y_{2})^{2}}, \text{ for } x_1,x_2 \in \MM \text{ and } y_1,y_2 \in \EE.\end{equation*}

 It is well known that 
$$\text{dim}_{\text{H}}(\MM\times\EE )\geq \text{dim}_{\text{H}}(\MM)+\text{dim}_{\text{H}}(\EE)  $$ (see \cite{Falconer}, Chapter 7). In Proposition \ref{prdoctv} we will prove the analog result for mean Hausdorff dimension. We  will   use the next lemmas. 

\begin{lemma}\label{mass}
Let $(\MM,d)$ be a compact metric space and $\varepsilon>0$. Suppose   there is a Borel measure $\mu$ on $(\MM,d)$ such that $\mu(\MM) \geq 1$ and for any open  ball  $E_{i}$   with $\emph{diam}_d E_{i} \leq \varepsilon$, we have that
$$
    \mu(E_{i}) \leq (\emph{diam}_d (E_{i}))^{{s}} \quad\text{for any }i\geq 1.
$$
Then, 
$$\emph{dim}^{\star}_{\emph{H}}(\MM,d,\varepsilon) \geq  {s}.$$
\end{lemma}

\begin{proof}
Fix  $\varepsilon > 0$   and take  a finite cover $\lbrace E_i \rbrace_{i =1}^{m}$  of $\MM$, by balls $E_{i}$  with   $\text{diam}_d (E_i) \leq \varepsilon$. We have that
 \begin{equation}\label{vevezwwww}
   \Sigma_{i =1}^{m} (\text{diam}_d(E_{i}))^{s}
   \geq   \Sigma_{i=1}^{m} \mu(E_{i}) \geq   \mu(\cup_{k =1}^{m} E_{i}) =  \mu (\MM) = 1.
\end{equation}
Hence, 
$\text{B}^{{s}}_{\varepsilon}(\MM,d) \geq 1$  and therefore $ \text{dim}^{\star}_{\text{H}}(\MM,d,\varepsilon) \geq {s}$ (see Lemma \ref{fege6}).
\end{proof}

The Lemma \ref{mass} is an adaption of the  \textit{Mass Distribution Principle}  (see \cite{Falconer}, Chapter 4), which states that if  there is a mass distribution   $\mu$ on $(\MM,d)$ 
 and for some  $s$ there are numbers $c>0$ and $\varepsilon>0$ such that  $\mu(E_{i})\leq c(\text{diam}_d (E_{i}))^{s}$ for any set $E_{i}$ with $\text{diam}_d (E_{i}) \leq \varepsilon$, we have that
$$\text{dim}_{\text{H}}(\MM,d) \geq  {s}.$$   We choose the version in Lemma \ref{mass}, because it is more compatible with the definition of mean Hausdorff dimension used in this work.

\begin{lemma}\label{frostman}
Let $c\in (0,1)$. There exists $ \varepsilon_0 =\varepsilon_0(c)\in(0,1)$ depending only on $c$ and  such that: for any compact metric space $(\MM, d)$ and $0 < \varepsilon \leq \varepsilon_0$ there exists a Borel probability measure $\mu$ on $(\MM,d)$ such that
$$\mu(E) \leq (\emph{diam}_d (E))^{c \emph{dim}_\emph{H}(\MM,d,\varepsilon)} $$
for all $E \subset \MM$ with $\emph{diam}_d (E) < \frac{\varepsilon}{6}.$
\end{lemma}

\begin{proof}
See \cite{LT2019}, Lemma 4.5.
\end{proof}
\begin{proposition}\label{prdoctv}  Take two continuous maps $f:\MM\rightarrow \MM$ and $g:\EE\rightarrow \EE$. On $\MM\times \EE$   consider the metric given in \eqref{bnm}.  We have:
$$\underline{\emph{mdim}}_{\emph{H}}(\MM \times \EE,  d \times d^{\prime}, f \times g) \geq \underline{\emph{mdim}}_{\emph{H}}(\MM, d, f) + \underline{\emph{mdim}}_{\emph{H}}(\EE, d' , g).$$
\end{proposition}

\begin{proof} First, we will prove for any $0 < c < 1$ there is $\delta_0 = \delta_0(c) \in (0,1)$ such that, for all $\delta \in (0, \delta_0]$, we have
$$\text{dim}_{\text{H}}(\MM \times \EE,d \times d^{\prime}, \delta / 6) \geq c ( \text{dim}_{\text{H}}(\MM,d,\delta) + \text{dim}_{\text{H}}(\EE,d^{\prime},\delta)).$$ Fix $0< c < 1$. It  follows  from Lemma \ref{frostman} that there is $\delta_0 = \delta_0(c) \in (0,1)$ such that for all $\delta \in (0, \delta_0]$ there are Borel probabilities measures $\mu$ and $\nu$ in $(\MM,d)$ and $(\EE,d')$, respectively, satisfying 
$$\mu(M) \leq (\text{diam}_d(M))^{c \text{dim}_{\text{H}}(\MM,d,\delta)}\quad\text{
and }\quad \nu(E) \leq (\text{diam}_{d^{\prime}}(E))^{c \text{dim}_{\text{H}}(\EE,d^{\prime},\delta)}$$
for all $M \subset \MM$ and $E \subset \EE$ with $\text{diam}_d (M) < \frac{\delta}{6}$ and  $\text{diam}_{d^{\prime}} (E) < \frac{\delta}{6}$. Observe that
$$\text{diam}_{d\times d^{\prime}}(M \times E) \geq \max(\text{diam}_d (M), \text{diam}_{d^{\prime}}(E)).$$
If $B$ is a ball in $\MM \times \EE$  with the metric \eqref{bnm}, then $B=M\times E$, where $M\subseteq \MM$ and $E\subseteq \EE$.  Next, for all $M \times E \subseteq \MM \times \EE$ such that $\text{diam}_{d \times d'}(M \times E) < \frac{\delta}{6}$, we have
    \begin{align*}
    (\mu \times \nu) (M\times E) & = \mu(M) \nu(E)  \leq (\text{diam}_d(M))^{c \text{dim}_{\text{H}}(\MM,d,\delta)}(\text{diam}_{d'}(E))^{c \text{dim}_{\text{H}}(\EE,d',\delta)}\\
    & \leq (\text{diam}_{d\times d'}(M \times E))^{c \text{dim}_{\text{H}}(\MM,d,\delta)}(\text{diam}_{d\times d'}(M \times E))^{c \text{dim}_{\text{H}}(\EE,d',\delta)}\\
    & = (\text{diam}_{d\times d'}(M \times E))^{c( \text{dim}_{\text{H}}(\MM,d,\delta) + \text{dim}_{\text{H}}(\EE,d',\delta))}.
    \end{align*}
By Lemma \ref{mass}, we get $$\text{dim}_{\text{H}}(\MM \times \EE,d \times d^{\prime}, \delta / 6) \geq c ( \text{dim}_{\text{H}}(\MM,d,\delta) + \text{dim}_{\text{H}}(\EE,d^{\prime},\delta)).$$

Next, for each $k\geq 1$, take  $ c_{k} \in (0,1)$ such that $c_{k}\rightarrow 1$ as $k\rightarrow \infty$. It  follows from the above fact  there is a $\delta_{k}(c_{k}) = \delta_k \in (0,1)$ such that $\delta_{k}\rightarrow 0$ as $k\rightarrow \infty$ and 
$$\text{dim}_{\text{H}}(\MM \times \EE,(d \times d^{\prime})_n, \delta_{k}/6) \geq c_{k} ( \text{dim}_{\text{H}}(\MM,d_n,\delta_{k}) + \text{dim}_{\text{H}}(\EE,d'_n,\delta_{k})),$$
for all $ n ,k\in \NN$. Hence, for each $k, n$,  we have 
$$
\frac{1}{n}\text{dim}_{\text{H}}(\MM \times \EE,(d \times d')_n, \delta_{k}/6) \geq \frac{c_{k}}{n} ( \text{dim}_{\text{H}}(\MM,d_n,\delta_{k}) + \text{dim}_{\text{H}}(\EE,d'_n,\delta_{k})) .
$$

Therefore,  taking  the  limit infimum as   $n \to \infty$ and the limit as $k \to \infty$, we have $$\underline{\text{mdim}}_{\text{H}}(\MM \times \EE , d \times d^{\prime}, f \times g) \geq \underline{\text{mdim}}_{\text{H}}(\MM, d, f) + \underline{\text{mdim}}_{\text{H}}(\EE , d^{\prime}, g),$$ 
which proves the result.
\end{proof}

Let $\mathbb{K}=\mathbb{N}$ or $\mathbb{Z}$. For $\bar{x} = (x_k),\bar{y}=(y_k) \in \MM^{\mathbb{K}}$, set 
 \begin{equation}\label{metrionmk} {\textbf{d}}(\bar{x},\bar{y}) = \underset{j \in\mathbb{K}}{\sum} \frac{1}{2^{|j|}}d(x_k,y_k).\end{equation}  Let $\sigma:\MM^{\mathbb{K}} \rightarrow \MM^{\mathbb{K}}  $ be the left shift map. In \cite{lind}, it is proved that $$    {\text{mdim}}(\MM^{\mathbb{K}} ,\sigma) \leq  {\text{dim}}(\MM).$$
In \cite{VV}, it is proved that $$    \underline{\text{mdim}}_{\text{M}}(\MM^{\mathbb{K}} ,\text{\textbf{d}},\sigma) =\underline{\text{dim}}_{\text{B}}(\MM,d)\quad\text{and}\quad \overline{\text{mdim}}_{\text{M}}(\MM^{\mathbb{K}} ,\text{\textbf{d}},\sigma) = \overline{\text{dim}}_{\text{B}}(\MM,d).$$
 We address these facts for the case of the mean Hausdorff dimension. We will need the following lemma:

\begin{lemma}\label{czzzvrv} Let $\sigma:\MM^{\mathbb{K}} \rightarrow \MM^{\mathbb{K}}  $ be the left shift map, with $\mathbb{K}=\mathbb{N}$ or $\mathbb{Z}$. Let $\mathcal{T}$  be the set consisting of all  finite open cover $\{C_{i}\}_{i=1}^{m}$ 
of $\MM^{\mathbb{K}}  $, such that each $C_{i}$ has  the form $C_{i} = A_{i,1} \times A_{i,2} \times \dots \times A_{i,\beta} \times \MM \times \MM \times \cdots$ and $A_{i,j}$ is an open subset of $\MM$, for $i=1,\dots, m$, $j=1,\dots,\beta$.  For every  $s\geq0$ and $\varepsilon >0$,  set \begin{equation*}\label{varepsccc}\emph{P}_{\varepsilon}^{s} (\MM^{\mathbb{N}},\emph{\textbf{d}}_{n})=\underset{\{C_{i}\}_{i=1}^{m}\in \mathcal{T}}{\inf}\left\{ \Sigma_{i=1}^{m}(\emph{diam}_{\emph{\textbf{d}}_{n}} ( C_{n}))^{s}: \MM^{\mathbb{K}}=\cup_{i=1}^{m} C_{i} \text{ with } \emph{diam}_{\emph{\textbf{d}}_{n}} (C_{i})<\varepsilon\right\}. \end{equation*}  Let $\Theta >0$ and set $$ \emph{dim}^{\bullet}_{\emph{H}}(\MM,\emph{\textbf{d}}_{n},\varepsilon,\Theta)=\sup\{s \geq 0:  \emph{P}_{\varepsilon}^{s} (\MM^{\mathbb{N}},\emph{\textbf{d}}_{n}) \geq \Theta\}.  $$ We have that
\begin{equation}\label{qwwsxccc}\underline{\emph{mdim}}_{\emph{H}}(\MM^{\mathbb{K}} ,\emph{\textbf{d}},\sigma) =\lim_{\varepsilon\rightarrow 0} \left(\liminf_{n\rightarrow \infty}\frac{1}{n} \emph{dim}^{\bullet}_{\emph{H}}(\MM,\emph{\textbf{d}}_{n},\varepsilon,\Theta)\right)\end{equation} and 
\begin{equation}\label{qwsxccc}\overline{\emph{mdim}}_{\emph{H}}(\MM^{\mathbb{K}} ,\emph{\textbf{d}},\sigma) =\lim_{\varepsilon\rightarrow 0} \left(\limsup_{n\rightarrow \infty}\frac{1}{n} \emph{dim}^{\bullet}_{\emph{H}}(\MM,\emph{\textbf{d}}_{n},\varepsilon,\Theta)\right).\end{equation}
\end{lemma}
\begin{proof}
    Clearly we have that $$\underline{\text{mdim}}_{\text{H}}(\MM^{\mathbb{K}} ,\text{\textbf{d}},\sigma) \geq \lim_{\varepsilon\rightarrow 0} \left(\liminf_{n\rightarrow \infty}\frac{1}{n} \text{dim}^{\bullet}_{\text{H}}(\MM,\text{\textbf{d}}_{n},\varepsilon,\Theta)\right)$$ and $$\overline{\text{mdim}}_{\text{H}}(\MM^{\mathbb{K}} ,\text{\textbf{d}},\sigma) \geq \lim_{\varepsilon\rightarrow 0} \left(\limsup_{n\rightarrow \infty}\frac{1}{n} \text{dim}^{\bullet}_{\text{H}}(\MM,\text{\textbf{d}}_{n},\varepsilon,\Theta)\right).$$ 

Next, we can prove that \begin{equation*}     \text{B}_{\varepsilon}^{s} (\MM^{\mathbb{K}},\text{\textbf{d}}_{n})\leq  2^{s}\text{P}_{{\varepsilon}/{2}}^{s} (\MM^{\mathbb{K}},\text{\textbf{d}}_{n}) \end{equation*} (see \cite{Falconer}, Section 2.4). From this fact (see \eqref{xfddefwfwfzx}), we can show that there exists $\Theta >0$ such that    
  $$   {\text{dim}}^{\star}_{\text{H}}(\MM^{\mathbb{K}},\text{\textbf{d}}_{n},\varepsilon)\leq  \dim_{\text{H}}^{\bullet}(\MM^{\mathbb{K}},\text{\textbf{d}}_{n},\varepsilon/2,\Theta) .$$  

From the above results, we have that  \eqref{qwwsxccc} and \eqref{qwsxccc} are valid for any  $\Theta >0$. 
\end{proof}

\begin{theorem}\label{theo1} Let $\sigma:\MM^{\mathbb{K}} \rightarrow \MM^{\mathbb{K}}  $ be the left shift map, with $\mathbb{K}=\mathbb{N}$ or $\mathbb{Z}$.  For any metric $d$ on $\MM$, we have that 
$$ \emph{dim}_{\emph{H}}(\MM,d)\leq \underline{\emph{mdim}}_{\emph{H}}(\MM^{\mathbb{K}} ,\emph{\textbf{d}},\sigma) \leq  \overline{\emph{mdim}}_{\emph{H}}(\MM^{\mathbb{K}} ,\emph{\textbf{d}},\sigma) \leq \underline{\emph{dim}}_{\emph{B}}(\MM,d).$$
\end{theorem}

\begin{proof} The second inequality is immediate from the definition. Next, in \cite{VV} it is proved that $\underline{\text{mdim}}_{\text{M}}(\MM^{\mathbb{K}} ,\text{\textbf{d}},\sigma) =\underline{\text{dim}}_{\text{B}}(\MM,d)$. Hence, the third inequality from the theorem follows from the fact that $ \overline{\text{mdim}}_{\text{H}}(\MM^{\mathbb{K}} ,{\textbf{d}},\sigma) \leq\underline{\text{mdim}}_{\text{M}}(\MM^{\mathbb{K}} ,\text{\textbf{d}},\sigma) $ (see Remark \ref{obsscc}).  
 
  We will prove the first inequality for $\mathbb{K}=\mathbb{N}$ (the case $\mathbb{K}=\mathbb{Z}$ can be proved analogously).  For each $k\geq 1$, take  $ c_{k} \in (0,1)$ such that $c_{k}\rightarrow 1$ as $k\rightarrow \infty$. It follows from  Lemma \ref{frostman} that, for each $k\geq 1$,    there exists a $\delta_k = \delta_k(c_k) \in (0,1)$, such that $\delta_{k}\rightarrow 0$ as $k\rightarrow \infty$, for which there is a Borel probability measure $\mu$ on $(\MM,d)$ such that
$$\mu(E) \leq (\text{diam}_d (E))^{c_k\text{dim}_{\text{H}}(\MM,d,\delta_k)}$$
for all $E \subset \MM$ with $\text{diam}_d (E) < \frac{\delta_k}{6}.$

Next, we will consider  the Borel probability measure $\tilde{\mu} = {\mu}^{\NN}$   on $\MM^{\NN}$. Let  $  \{C_{i}\}_{i=1}^{m}$ be a finite  open cover  of $\MM^{\mathbb{N}}$ with the form $C_{i} = A_{i,1} \times A_{i,2} \times \dots \times A_{i,\beta} \times \MM \times \MM \times \cdots,$ where $A_{i,j}  $ is an open subset of $\MM$, for all $1 \leq j \leq \beta.$ We will suppose that  $\text{diam}_{\textbf{d}_{n}} (C_{i})< \frac{\delta_k}{6(2^{\beta})}$, for all $i=1,\dots,\beta$. In this case, we must have that $\text{diam}_{{d}} (A_{i,j})< \frac{\delta_k}{6}$, for $i=1,\dots, m$, $j=1,\dots,\beta$  and $\beta \gg n$. 
 Therefore, for all $C_{i}$,  we have that 
\begin{align*}
    \tilde{\mu}(C_i) & = \mu(A_{i,1}) \mu(A_{i,2}) \cdots \mu(A_{i,\beta}) \leq (\text{diam}_d (A_{i,1}))^{c_k\,\text{dim}_{\text{H}}(\MM,d,\delta_k)} \cdots (\text{diam}_d (A_{i,\beta}))^{c_k\,\text{dim}_{\text{H}}(\MM,d,\delta_k)}\\
    & \leq  (\text{diam}_d (A_{i,1}))^{c_k\,\text{dim}_{\text{H}}(\MM,d,\delta_k)} \cdots (\text{diam}_d (A_{i,n}))^{c_k\,\text{dim}_{\text{H}}(\MM,d,\delta_k)}\\
    &\leq (\text{diam}_{\textbf{d}_n} (C_i))^{c_k n\, \text{dim}_{\text{H}}(\MM,d,\delta_k)}.
\end{align*}
From this fact, we can to prove   that
$$\frac{1}{n}\text{dim}_{\text{H}}^{\bullet}(\MM^{\NN},\textbf{d}_n, {\delta_k}/{6(2^{\beta})}) \geq c_k \text{dim}_{\text{H}}(\MM,d,\delta_k)$$ (see \eqref{vevezwwww}), 
where $c_k \to 1$ and $\delta_k \to 0$ as $k \to \infty$. The theorem follows from Lemma   \ref{czzzvrv}. \end{proof}

For mean topological dimension   we have  that $$   {\text{mdim}} (N^{\mathbb{Z}} ,\sigma)\leq \text{dim}(N)   ,$$ where $\text{dim}(N)$ is the topological dimension of $N$   (see \cite{lind}, Theorem 3.1). This inequality can be strict (see  \cite{Tsukamoto2}). 

\medskip 

\noindent \textbf{Conjecture.} We conjecture that for any compact metric space $\MM$ we have that $$  \underline{\text{mdim}}_{\text{H}}(\MM^{\mathbb{K}} ,\text{\textbf{d}},\sigma) =\text{dim}_{\text{H}}(\MM,d).$$

\medskip 
 
Next, for any continuous map $f \colon \MM \to \MM$, we have
$$\underline{\text{mdim}}_{\text{M}}(\MM,d,f) \leq \overline{\text{mdim}}_{\text{M}}(\MM,d,f) \leq \underline{\text{dim}}_{\text{B}}(\MM,d)$$ (see \cite{VV}). 
Consequently, from Remark \ref{obsscc}, we have $$\underline{\text{mdim}}_{\text{H}}(\MM,d,f) \leq \overline{\text{mdim}}_{\text{H}}(\MM,d,f) \leq \underline{\text{dim}}_{\text{B}}(\MM,d).$$

The next corollary follows from Theorem \ref{theo1}.

\begin{corollary}  Suppose that $\emph{dim}_{\emph{H}}(\MM,d)= \underline{\emph{dim}}_{\emph{B}}(\MM,d)$, then: \begin{itemize} \item $  \underline{\emph{mdim}}_{\emph{H}}(\MM^{\mathbb{K}} ,\emph{\textbf{d}},\sigma) =  \overline{\emph{mdim}}_{\emph{H}}(\MM^{\mathbb{K}} ,\emph{\textbf{d}},\sigma) ={\emph{dim}_{\emph{H}}}(\MM,d)$.
\item For any $f\in C^{0}(\MM)$ we have 
$ \underline{\emph{mdim}}_{\emph{H}}(\MM,d,f) \leq \overline{\emph{mdim}}_{\emph{H}}(\MM,d,f) \leq  {\emph{dim}_{\emph{H}}}(\MM,d).$\end{itemize}
\end{corollary}

   \section{Some examples changing the metric}\label{someexamples}

In this section, we will calculate the metric mean dimension of several continuous maps changing the metric on  $\MM$.  For any homeomorphism $h \colon \MM \to \MM$, take the metric $d_{h}\in\MM(\tau)$ defined by \begin{equation}\label{fedddzff} d_h(x,y) = d(h(x),h(y)) \quad\text{for all }x,y \in \MM.\end{equation}  

Next, take  $g:\MM\rightarrow \MM$ given by $g(x) = h \circ f \circ h^{-1}(x)$, for all $x \in \MM$, where $f:\MM\rightarrow \MM$ is a fixed continuous map.  We have that the map $h \colon (\MM,d_h) \to (\MM,d)$ is an isometry. Therefore,  for any homeomorphism $h:\MM\rightarrow \MM$ we have
$$\text{mdim}_{\text{M}}(\MM,d_h,f) = \text{mdim}_{\text{M}}(\MM,d, h \circ f \circ h^{-1}) = \text{mdim}_{\text{M}}(\MM,d,g)$$
and
$$\text{mdim}_{\text{H}}(\MM,d_h, f) = \text{mdim}_{\text{H}}(\MM ,d,h \circ f \circ h^{-1}) = \text{mdim}_{\text{H}}(\MM,d,g).$$
Consequently,  
$$ \text{mdim}_{\text{M}}(\MM,d_h,f) \in [0, \text{dim}_{\text{B}}(\MM,d) ]\quad\text{ and }\quad  \text{mdim}_{\text{H}}(\MM,d_h, f)  \in [0,  \text{dim}_{\text{B}}(\MM,d)].   $$

Since the metric mean dimension depends on the metric, we can have two topologically conjugate dynamical systems with different metric mean dimension, as we will see in the next example (see   \cite{Muentes}, \cite{Kloeckner} and \cite{VV}).

 \begin{example}\label{example11} For any closed interval $J$, let $T_{J}: J \rightarrow [0,1]$ be the unique increasing affine map from $J$ onto $[0,1]$. Set $g(x) = |1-|3x-1||$ for any $x \in [0,1]$. Fix $r\in(0,\infty)$ and $s\in \mathbb{N}$.

For any $n\geq 1$, set $a_{0}=0$,  $a_{n}= \sum_{i=0}^{n-1}\frac{A}{3^{ir}}$ and take $I_{n}=[a_{n-1},a_{n}]$, where $A=\frac{1}{\sum_{i=0}^{\infty}\frac{1}{3^{ir}}}=\frac{3^{r}-1}{3^{r}}$. Next, take $\phi_{s,r}\in C^{0}([0,1])$, given by   $\phi_{s,r}|_{I_{n}}= T_{I_n}^{-1}\circ g^{s n}\circ T_{I_n}$ for any $n\geq 1$.  We have (see  \cite[Example 2.5]{Muentes3}, \cite[Example 3.1]{Muentes} and \cite[Lemma 6]{VV})  $${{\text{mdim}}_{\text{H}}}([0,1] ,|\cdot |,\phi_{s,r})=  {{\text{mdim}}_{\text{M}}}([0,1] ,|\cdot |,\phi_{s,r}) =\frac{s}{r+s}.$$  
 For a fixed $s$ and any $r_{1}, r_{2}\in (0,\infty)$, we have  $\phi_{s,r_{1}}$ and $\phi_{s,r_{2}}$ are topologically conjugate by a conjugacy $h_{1,2}:[0,1]\rightarrow [0,1]$ (see \cite{Muentes}, Remark 3.2), such that $$  \phi_{s,r_{1}}=h_{1,2}  \circ\phi_{s,r_{2}} \circ h_{1,2} ^{-1}.$$  Hence, 
 $$\text{mdim}_{\text{M}}([0,1] ,d_{h_{1,2}},\phi_{s,r_{2}})    =\frac{s}{r_{1}+s}\neq \frac{s}{r_{2}+s}=\text{mdim}_{\text{M}}([0,1] ,|\cdot|,\phi_{s,r_{2}}) ,$$  where $d_{h_{1,2}}$ is defined in \ref{fedddzff}. The same fact holds for the mean Hausdorff dimension.
 
 Next, for $n\geq 1$, set   $J_{n}=[2^{-n^{n}},2^{-{n}^{n}+1}]$.
  Take $\varphi_{s}\in C^{0}([0,1])$, given by   $\varphi_{s}|_{J_{n}}= T^{-1}_{J_{n}}\circ g^{s n}\circ T_{J_n}$ for any $n\geq 1$. We can prove   that  $$ \text{mdim}_{\text{H}}([0,1],|\cdot |, \varphi_{s})=\text{mdim}_{\text{M}}([0,1],|\cdot |, \varphi_{s})=0$$  (see  \cite[Theorem 3.3]{Muentes}). 
 Note that, for any $s\in \mathbb{N}$ and $r\in (0,\infty)$,  $\varphi_{s}$ and $\phi_{r,s}$ are topologically conjugate  by a topological conjugacy $h:[0,1]\rightarrow [0,1]$ such that $ \varphi_{s}= h\circ \phi_{s,r}\circ h^{-1}$. Hence,  $$ \text{mdim}_{\text{H}}([0,1] ,d_{h},\phi_{s,r})= \text{mdim}_{\text{M}}([0,1] ,d_{h},\phi_{s,r})= \text{mdim}_{\text{M}}([0,1] ,|\cdot |,\varphi_{s})=0.   $$
  
Finally, let $b_{0}=0$ and     $b_{n}=\sum_{i=1}^{n}\frac{6}{\pi^{2}i^{2}}$ for any $n\geq 1.$  Take $K_{n}=[b_{n-1},b_{n}]$. Let $\psi_{s}\in C^{0}([0,1])$ be defined by  $\psi_{s} |_{K_{n}} = T_{K_n}^{-1}\circ g^{sn}\circ T_{K_n}$ for any $n\geq 1$. 
 We have that (see  \cite[Example 3.5]{Muentes} and \cite[Example 2.6]{Muentes3})  $$ \text{mdim}_{\text{H}}([0,1],|\cdot |, \psi_{s})=\text{mdim}_{\text{M}}([0,1],|\cdot |, \psi_{s})=1.$$   Note that, for any $s\in \mathbb{N}$ and $r\in (0,\infty)$,  $\psi_{s}$ and $\phi_{r,s}$ are topologically conjugate  by a topological conjugacy $j:[0,1]\rightarrow [0,1]$ such that $ \psi_{s}= j\circ \phi_{s,r}\circ j^{-1}$. Hence,  $$ \text{mdim}_{\text{H}}([0,1] ,d_{j},\phi_{s,r})=  \text{mdim}_{\text{M}}([0,1] ,d_{j},\phi_{s,r})= \text{mdim}_{\text{M}}([0,1] ,|\cdot |,\psi_{s})=1.   $$
  \end{example}

\begin{remark}\label{wfwfaffa} Let  $\mathcal{M}$ be the subset of $C^{0}([0,1])$ consisting of each map $f$ such that for some closed subinterval $K\subseteq [0,1]$, $f|_{K}:K\rightarrow K$ is such that $f=T_{K}^{-1}\circ \psi\circ T_{K}$, where $\psi$ is one of the maps defined in Example \ref{example11} (that is, $\phi_{s,r}$, or $\varphi_{s}$, or $\psi_{s}$),  and  $f|_{K^{c}}:K^{c}\rightarrow K^{c}$ is a piecewise $C^{1}$-map.   $\mathcal{M}$  is dense in $C^{0}([0,1])$ (see \cite{Carvalho} and \cite{Muentes}). Note that for each $f \in \mathcal{M}$ and $a \in [0,1]$, based on Example \ref{example11}, it is possible to construct an explicit metric $d_{a}$ on $[0,1]$ such that $$\text{mdim}_{\text{H}}([0,1],d_{a},f)=\text{mdim}_{\text{M}}([0,1],d_{a},f)=a.$$\end{remark}
  
 \begin{remark}\label{observacr}
 In Example \ref{example11}, note that $\phi_{s,r}=\phi_{1,r}^{s}$ for any $s\in\mathbb{N}$ and $r\in(0,\infty)$. Hence, 
 $$  \text{mdim}_{\text{M}}([0,1] ,|\cdot|,\phi^{s}_{1,r})=\frac{s\, \text{mdim}_{\text{M}}([0,1] ,|\cdot|,\phi_{1,r})}{1+(s-1)\text{mdim}_{\text{M}}([0,1] ,|\cdot|,\phi_{1,r})}.  $$ The same fact holds for the mean Hausdorff dimension. 
 \end{remark}

     Let  \begin{equation*}\label{nevsuidnf} \textit{\textbf{C}}=\{(x_{1},x_{2},\dots ): x_n=0,2\text{  for }n\in\mathbb{N}\}=\{0,2\}^{\mathbb{N}}\end{equation*}
be the Cantor set.    For a fixed $\alpha \in (1,\infty)$, consider the metric    
\begin{equation*}\label{gsgdf} \textbf{d}_{\alpha}(\bar{x},\bar{y})=\sum_{n\in\mathbb{N}} \alpha^{-n}|x_{n}-y_{n}|,\quad\text{for any }\bar{x}=(x_{n})_{n\in\mathbb{N}},\bar{y}=(y_{n})_{n\in\mathbb{N}}\in \textit{\textbf{C}} . \end{equation*}   We have that  $\text{dim}_{\text{B}}(\textit{\textbf{C}},\textbf{d}_{\alpha})= \frac{\log 2}{\log \alpha}$ (see \cite{Furstenberg}, Proposition III.1 or \cite{Falconer}, page  31). Therefore,  for any $\varphi\in C^{0}(\textit{\textbf{C}}), $ we have from Remark \ref{obsscc} that  \begin{equation}\label{iuhugltr}  \underline{\text{mdim}}_\text{M}({\textit{\textbf{C}}} ,\textbf{d}_{\alpha},\varphi)\leq \overline{\text{mdim}}_\text{M}({\textit{\textbf{C}}} ,\textbf{d}_{\alpha},\varphi)\leq  \text{dim}_{\text{B}}(\textit{\textbf{C}},\textbf{d}_{\alpha})=\frac{\log 2}{\log \alpha}.\end{equation}  
For any $k\geq 1$,    set  \begin{align*}\textit{\textbf{C}}_{k}=\{(x_{n})_{n=1}^{\infty}: x_{i}=0 \text{ for }i\leq k-1, x_{k}=2\text{ and } x_{n}\in \{0,2\}\text{ for }n\geq k+1\}.\end{align*} Note that if $k\neq s$, then $ \textit{\textbf{C}}_{k}\cap \textit{\textbf{C}}_{s}=\emptyset $ and ${\textit{\textbf{C}}\setminus \cup_{k=1}^{\infty} \textit{\textbf{C}}_{k}}=\{(0,0,\dots)\}$.  Furthermore,   each  $\textit{\textbf{C}}_{k}$ is a clopen subset  homeomorphic to $\textit{\textbf{C}}$ via the  homeomorphism 
$$T_{k}:\textit{\textbf{C}}_{k}\rightarrow \textit{\textbf{C}},\quad (\underset{(k-1)\text{-times}}{\underbrace{{0},\dots, {0}}},2,x_{1},x_{2},\dots)\mapsto (x_{1},x_{2},\dots),$$  
which is Lipschitz.

  \begin{example}\label{gshfkfwee} For $j \in\mathbb{N}$, consider $\psi_{j}:(\textit{\textbf{C}},\textbf{d}_{\alpha})\rightarrow (\textit{\textbf{C}},\textbf{d}_{\alpha})$  defined   as $\psi_{{j}}(0,0,\dots)=(0,0,\dots)$  and $  \psi_{j}|_{\textit{\textbf{C}}_{k}}= T_{k}^{-1}\sigma^{jk} T_{k} $ for $k\geq 1$, where $\sigma:\textit{\textbf{C}}\rightarrow \textit{\textbf{C}} $ is the left shift map.       In \cite{Muentes}, Proposition 5.1,   it is proven that if  $\alpha=3$, then   \begin{equation*} {\text{mdim}}_\text{M}({\textit{\textbf{C}}}   ,\textbf{d}_{3},\psi_{j})= \frac{j\log 2}{(j+1)\log 3}.\end{equation*}
Following the same steps, we  will  prove that  \begin{equation*}{\text{mdim}}_\text{M}({\textit{\textbf{C}}}  ,\textbf{d}_{\alpha} ,\psi_{j})=  \frac{j\log 2}{(j+1)\log \alpha}\quad\text{for any }\alpha >1.\end{equation*}      
Take $\varepsilon >0$. For any $k\geq 1,$ set    $\varepsilon_{k}=\alpha^{-k(j+1)}$. There exists $k\geq 1$ such that $\varepsilon\in[\varepsilon_{k+1},\varepsilon_{k}]$. 
For $n\geq 1$ and $k\geq 1$, take $\bar{z}_{1}=(z_{1}^{1},\dots,z_{jk}^{1}),\dots, \bar{z}_{n}=(z_{1}^{n},\dots,z_{jk}^{n})$, with $z_{i}^{s}\in \{0,2\},$ and set
$$ {A}^{k}_{\bar{z}_{1},\dots,\bar{z}_{n}}=\{(\underset{(k-1)\text{-times}}{\underbrace{{0},\dots, {0}}},2,z_{1}^{1},\dots,z_{jk}^{1},\dots,z_{1}^{n},\dots,z_{jk}^{n}, x_{1},\dots,x_s,.\,.\,.  ): x_{i}\in\{0,2\}\}\subseteq \textit{\textbf{C}}_{k}.$$
Note that if ${A}^{k}_{\bar{z}_{1},\dots,\bar{z}_{n}}\neq {A}^{k}_{\bar{w}_{1},\dots,\bar{w}_{n}}$ and $\bar{x}\in {A}^{k}_{\bar{z}_{1},\dots,\bar{z}_{n}}$, $\bar{y}\in {A}^{k}_{\bar{w}_{1},\dots,\bar{w}_{n}}$, then  $(\textbf{d}_{\alpha})_{n+1}(\bar{x},\bar{y})>\frac{1}{\alpha^{k(j+1)}}$.   Therefore, 
$  \text{sep}(n+1,\psi_{j},\varepsilon_{k}) \geq 2^{jnk}$ and hence 
\begin{align*}\limsup_{n\rightarrow \infty} \frac{\log\text{sep}(n+1,\psi_{j},\varepsilon) }{n+1}\geq 
\limsup_{n\rightarrow \infty} \frac{\log\text{sep}(n+1,\psi_{j},\varepsilon_{k}) }{n+1}& \geq \lim_{n\rightarrow \infty} \frac{n\log(2^{jk})}{n+1}= \log 2^{jk}.\end{align*}
Thus, 
\begin{align*}  \underline{\text{mdim}}_\text{M}({\textit{\textbf{C}}} ,\textbf{d}_{\alpha},\psi_{j})& \geq
\lim_{k\rightarrow \infty} \frac{\log\text{sep}(\psi_{j},\varepsilon_{k}) }{-\log \varepsilon_{k+1}} \geq  \lim_{k\rightarrow \infty}\frac{\log(2^{jk})}{\log (\alpha^{(k+1)(j+1)})} = \lim_{k\rightarrow \infty} \frac{{k}j\log2}{(k+1)(j+1)\log \alpha}\\
&=\frac{j\log 2}{(j+1)\log \alpha}.\end{align*}
Therefore,  \begin{equation}\label{iuhugl}  \overline{\text{mdim}}_\text{M}({\textit{\textbf{C}}} ,\textbf{d}_{\alpha},\psi_{j})\geq \underline{\text{mdim}}_\text{M}({\textit{\textbf{C}}} ,\textbf{d}_{\alpha},\psi_{j})\geq  \frac{j\log 2}{(j+1)\log \alpha}.\end{equation}

On the other hand,   note that for each $l\in \{1,\dots,k\}$, the sets $ {A}^{l}_{\bar{z}_{1},\dots,\bar{z}_{n}}$ have $(\textbf{d}_{\alpha})_{n}$-diameter less than
 $\varepsilon_k$. Furthermore, the sets $\{(0,0,\dots)\}$ and  $\overset{\infty}{\underset{s=k+1}{\bigcup}}  \textbf{\textit{C}}_{s}$ has $(\textbf{d}_{\alpha})_{n}$-diameter less than
 $\varepsilon_k$. Hence $$\text{cov}(n , \psi_{j},\varepsilon_{k})\leq k2^{njk} +2 \leq2k2^{njk}$$ 
and therefore
\begin{equation*}\text{cov}( \psi_{j},\varepsilon_{k})\leq\lim_{n\rightarrow\infty} \frac{\log(2k2^{njk})}{n}= \log 2^{jk}. \end{equation*} Hence  \begin{equation}\label{iuhugl2}  \overline{\text{mdim}}_\text{M}({\textit{\textbf{C}}}   ,\textbf{d}_{\alpha},\psi_{j})=\limsup_{\varepsilon\rightarrow 0}\frac{\text{cov}(\psi_{j},\varepsilon)}{-\log\varepsilon}\leq \limsup_{k\rightarrow \infty}\frac{\text{cov}(\psi_{j},\varepsilon_{k+1})}{-\log\varepsilon_{k}} \leq \frac{j\log 2}{(j+1)\log \alpha}.\end{equation}
It follows from \eqref{iuhugl} and \eqref{iuhugl2} that \begin{equation*}   {\text{mdim}}_\text{M}({\textit{\textbf{C}}}   ,\textbf{d}_{\alpha},\psi_{j})= \frac{j\log 2}{(j+1)\log \alpha}.\end{equation*}
\end{example}

\begin{example}\label{gshfkf}  
Take $\varphi:(\textit{\textbf{C}},\textbf{d}_{\alpha})\rightarrow (\textit{\textbf{C}},\textbf{d}_{\alpha})$ the map defined   as $\varphi(0,0,\dots)=(0,0,\dots)$  and $  \varphi|_{\textit{\textbf{C}}_{k}}= T_{k}^{-1}\sigma^{{k}^{2}} T_{k} $ for $k\geq 1$, where $\sigma:\textit{\textbf{C}}\rightarrow \textit{\textbf{C}} $ is the left shift map.    Note that $\varphi$ is a continuous map.  We prove that $$   {\text{mdim}}_\text{M}({\textit{\textbf{C}}} ,\textbf{d}_{\alpha},\varphi)=  \text{dim}_{\text{B}}(\textit{\textbf{C}},\textbf{d}_{\alpha})=\frac{\log 2}{\log \alpha}.$$ 
Take $\varepsilon >0$. For any $k\geq 1,$ set    $\varepsilon_{k}=\frac{1}{\alpha^{{k}^{2}+k}}$. There exists $k\geq 1$ such that $\varepsilon\in[\varepsilon_{k+1},\varepsilon_{k}]$. 
For $n\geq 1$ and $k\geq 1$, take $\bar{z}_{1}=(z_{1}^{1},\dots,z_{{k}^{2}}^{1}),\dots, \bar{z}_{n}=(z_{1}^{n},\dots,z_{{k}^{2}}^{n})$, with $z_{i}^{s}\in \{0,2\},$ and set
$$ {A}^{k}_{\bar{z}_{1},\dots,\bar{z}_{n}}=\{(\underset{(k-1)\text{-times}}{\underbrace{{0},\dots, {0}}},2,z_{1}^{1},\dots,z_{{k}^{2}}^{1},\dots,z_{1}^{n},\dots,z_{{k}^{2}}^{n}, x_{1},\dots,x_s,.\,.\,.  ): x_{i}\in\{0,2\}\}\subseteq \textit{\textbf{C}}_{k}.$$
Note that if ${A}^{k}_{\bar{z}_{1},\dots,\bar{z}_{n}}\neq {A}^{k}_{\bar{w}_{1},\dots,\bar{w}_{n}}$ and $\bar{x}\in {A}^{k}_{\bar{z}_{1},\dots,\bar{z}_{n}}$, $\bar{y}\in {A}^{k}_{\bar{w}_{1},\dots,\bar{w}_{n}}$, then  $(\textbf{d}_{\alpha})_{n+1}(\bar{x},\bar{y})>\frac{1}{\alpha^{k^{2}+k}}$.   Therefore 
$  \text{sep}(n+1,\varphi,\varepsilon_{k}) \geq \left(2^{k^{2}}\right)^{n}$ and hence 
\begin{align*}\lim_{n\rightarrow \infty} \frac{\log\text{sep}(n+1,\varphi,\varepsilon) }{n+1}\geq 
\lim_{n\rightarrow \infty} \frac{\log\text{sep}(n+1,\varphi,\varepsilon_{k}) }{n+1}& \geq \lim_{n\rightarrow \infty} \frac{n\log(2^{{k}^{2}})}{n+1}= \log 2^{{k}^{2}}.\end{align*}
Thus, 
\begin{align*}  \underline{\text{mdim}}_\text{M}({\textit{\textbf{C}}} , \textbf{d}_{\alpha},\varphi)& \geq
\liminf_{k\rightarrow \infty} \frac{\log\text{sep}(\varphi,\varepsilon_{k}) }{-\log \varepsilon_{k+1}} \geq  \lim_{k\rightarrow \infty}\frac{\log(2^{k^{2}})}{\log (\alpha^{(k+1)^{2}+k+1})} \\
&= \lim_{k\rightarrow \infty} \frac{k^{2}\log2}{((k+1)^{2}+k+1)\log \alpha} = \frac{\log 2}{\log \alpha} .\end{align*}
Therefore, by \eqref{iuhugltr}, we have that  \begin{equation*}  \overline{\text{mdim}}_\text{M}({\textit{\textbf{C}}} ,\textbf{d}_{\alpha},\varphi)= \underline{\text{mdim}}_\text{M}({\textit{\textbf{C}}} ,\textbf{d}_{\alpha},\varphi)= \frac{\log 2}{\log \alpha}.\end{equation*}
\end{example}

\section{On the continuity of metric and Hausdorff mean dimension maps}\label{section5}

Throughout this section, we will work with a fixed metrizable compact topological space $(\MM, \tau)$. We use $\MM(\tau)$ to denote the set of all metrics that induce the same topology $\tau$ on $\MM$. Formally, this set is defined as:
$$\MM(\tau) = \lbrace d \colon d \text{ is a metric for } \MM \text{ and } \tau_d = \tau \rbrace,$$
where $\tau_{d}$ is the topology induced by   $d$ on $\MM$. We remember that two metrics on a space $\MM$ are equivalent if they induce the same topology on $\MM$. Therefore, if $d$ is a fixed metric on $\MM$ which induces the topology $\tau$, then $\MM(\tau)$ consists on all the metrics on $\MM$ which are equivalent to $d$. 

\medskip 

From now on, we will fix a continuous map $f:\MM\rightarrow \MM$.  Consider the functions   \begin{equation*} 
\begin{aligned}
\text{mdim}_{\text{M}}(\MM,f) \colon \MM(\tau)& \to \mathbb{R} \cup \lbrace \infty \rbrace\\
d& \mapsto \text{mdim}_{\text{M}}(\MM,d,f)
\end{aligned} \quad\text{and}\quad \begin{aligned}
\text{mdim}_{\text{H}}(\MM,f) \colon \MM(\tau) &\to \mathbb{R} \cup \lbrace \infty \rbrace\\
d& \mapsto \text{mdim}_{\text{H}}(\MM,d,f),
\end{aligned}
\end{equation*}
where $\MM(\tau)$  is endowed with the metric  
$$ D(d_1,d_2) = \max_{x,y \in \MM} \left\{| d_1(x,y) - d_2(x,y)|: \text{ for } d_1,d_2 \in \MM(\tau) \right\}$$ 
(see \cite{metric}).  We will prove   there exist continuous maps $f:\MM\rightarrow \MM$ such that   $ \text{mdim}_{\text{M}}(\MM,f)$ is not a continuous map.   
\begin{remark}\label{fefefeesee}
    Remember   that two metrics $d_1$ and $d_2$ on $\MM$ are  called \textbf{uniformly equivalent} if there are  real constants $0 < a \leq b$ such that
\begin{equation*}\label{fndjf} a d_1(x,y) \leq d_2(x,y) \leq b d_1(x,y),\end{equation*}
for all $x,y \in \MM$.  It is not difficult to see  that, if $d_1$ and $d_2  \in \MM(\tau)$ are two uniformly equivalent metrics on $\MM$, then
$$\text{mdim}_{\text{M}}(\MM , d_1,f) = \text{mdim}_{\text{M}}(\MM , d_2,  f)\quad\text{ 
   and 
}\quad\text{mdim}_{\text{H}}(\MM,d_1,f) =  \text{mdim}_{\text{H}}(\MM,d_2,f).$$  
\end{remark}
 
\begin{remark} Note  if $h_{\text{top}}(\MM,f) < \infty$, then $\text{mdim}_{\text{M}}(\MM,d,f)=0$. Therefore, as the topological entropy does not depend on
the metric, we have that $\text{mdim}_{\text{M}}(\MM,\tilde{d},f)=0$ for any $\tilde{d} \in \MM(\tau$). Analogously, we can prove that $\text{mdim}_{\text{H}}(\MM,\tilde{d},f)=0$ for any $\tilde{d} \in \MM(\tau$). Hence, if $h_{\text{top}}(\MM,f) < \infty$, then  $$\text{mdim}_{\text{M}}(\MM,f) \colon \MM(\tau) \to \mathbb{R} \quad \text{ and }\quad \text{mdim}_{\text{H}}(\MM,f) \colon \MM(\tau) \to \mathbb{R} $$ are  the zero maps. 
  \end{remark}

 In the next example, we will exhibit a class of dynamical systems such that the metric and  Hausdorff mean dimension maps are not continuous, with respect to the metric.

\begin{example}\label{exampledesc}   Take $\MM=[0,1]$ endowed with the metric $|\cdot |$  induced by the absolute value. For  fixed $s\in\mathbb{N}$ and $r\in (0,\infty)$, set $f=\phi_{s,r}:[0,1]\rightarrow [0,1]$  and $I_{n}=[a_{n-1},a_{n}]$ defined in Example \ref{example11}. Hence, $${\text{mdim}}_{\text{H}}([0,1] ,|\cdot |,f ) = {\text{mdim}}_{\text{M}}([0,1] ,|\cdot |,f ) =\frac{s}{r+s}.$$ Fix any metric $d$ on $ \MM$ equivalent to $|\cdot |$.  We will find two metrics ${d}_{1}$ and $d_{2}$ on $[0,1]$,  arbitrarily close to $d$, such that $$ {\text{mdim}}_{\text{M}}([0,1] ,{d}_{1},f ) =1\quad\text{and}\quad  {\text{mdim}}_{\text{M}}([0,1] , {d}_{2},f ) =\frac{1}{2}.$$ Let $\varepsilon >0$. There exists $N\in \mathbb{N}$ such that  $$\max\{\text{diam}_{d}(\cup_{n=N}^{\infty}I_{n}) \}<\frac{\varepsilon}{2}.$$ 
Set $b_{N}=a_{N} $ and $b_{n}=a_{N}+\sum_{j=1}^{n}\frac{6\varepsilon}{2\pi^{2} j^{2}}$ for $n\geq N+1$ and consider $J_{n}=[b_{n-1},b_{n}]$  for any $n\geq  N+1$. Take the homeomorphism $h:[0,1]\rightarrow [0,a_{N}+\frac{\varepsilon}{2}]$ defined by  
$$
h(x)= \begin{cases}  x\quad &\text{if } x  \in [0,a_{N}]\\
 a_{N}+\varepsilon/2\quad &\text{if } x =1\\
  \left[ \frac{b_{n+1}-b_{n}}{a_{n+1}-a_{n}}\right](x-a_{n})+b_{n}\quad & \text{if } x  \in I_{n}, \text{ for some } n\geq N+1.
\end{cases}$$
Consider the metric $d_{1}$ on $[0,1]$ given by 
$$d_{1}(x,y)= \begin{cases}  d(x,y)\quad &\text{if } x ,y \in [0,a_{N}]\\
  |h(x)-h(y)|\quad & \text{if } x,y  \in [a_{N},1]=\overline{\bigcup_{n= N+1}^{\infty}I_{n}}\\
    |h(x)-a_{N}|+ d(y,a_{N})\quad & \text{if }  y  \in [0,a_{N}], x \in [a_{N},1]\\
     |h(y)-a_{N}|+ d(x,a_{N})\quad & \text{if }  x \in [0,a_{N}], y \in [a_{N},1].
\end{cases}$$
As $d_1$ depends of the metric $d$ and of the homeomorphism $h$, we have that $d_1$ belongs to $\MM(\tau)$. Furthermore, 
$$ \text{diam}_{d_{1}}\left(\bigcup_{j=N+1}^{\infty} I_{j}\right)= \text{diam}_{|\cdot|}\left(\bigcup_{j=N+1}^{\infty} J_{j}\right) = \sum_{j=N+1}^{\infty}   |J_{j}|=  \sum_{j=N+1}^{\infty} \frac{6\varepsilon}{2\pi^{2} j^{2}}<\frac{\varepsilon}{2}. $$   We   prove that $D(d_{1},d)<\varepsilon$. If $x,y\in [0,a_{N}]$ or if $x,y\in [a_{N},1]$, then $|d(x,y)-d_{1}(x,y)|=0$. Suppose that $x\in [0,a_{N}]$ and $y\in [a_{N},1]$. From definition of $d_{1}$, we have that $$d_{1}(x,y)= |h(y)-a_{N}|+ d(x,a_{N}).$$ Since $d(x,y)\leq d(x,a_{N})+d(a_{N},y)$, it follows that
\begin{align*} d(x,y)-d_{1}(x,y)&\leq d(x,a_{N})+d(a_{N},y)-d(x,a_{N})-|a_{N}-h(y)|\\
&= d(a_{N},y)-|a_{N}-h(y)|<\varepsilon   \end{align*}
and 
\begin{align*} d_{1}(x,y)-d(x,y)  & =  d(x,a_{N})+|a_{N}-h(y)|- d(x,y)  \\
&\leq d(x,y)+d(y,a_{N})+|a_{N}-h(y)|- d(x,y) \\
&= d(y,a_{N})+|a_{N}-h(y)|  <\varepsilon.   \end{align*}
Hence,   $D(d_{1},d)<\varepsilon$.

Next, given that $h_{\text{top}}(f|_{[0,a_{N}]})<\infty,$ we have 
$${{\text{mdim}}_{\text{M}}}([0,1] ,d_{1},f ) ={{\text{mdim}}_{\text{M}}}([0,1] ,d_{1},f|_{[a_{N},a_{N}+\varepsilon/2]} ).$$ By  \cite[Example 3.1]{Muentes} and \cite[Example 2.6]{Muentes3}, it is possible to obtain that
$$ {{\text{mdim}}_{\text{H}}}([0,1] ,d_{1},f ) ={{\text{mdim}}_{\text{M}}}([0,1] ,d_{1},f ) =1.   $$ 

The existence of $d_{2}$ can be shown  analogously 
taking $r=s$, $c_{N}=a_{N}$ and $c_{n}=a_{N}+\sum_{j=1}^{n}\frac{A\varepsilon}{3^{is}}$ for $n\geq N+1$, where $A=\frac{1}{\sum_{j=1}^{\infty}{3^{-is}}}$, and considering $K_{n}=[c_{n-1},c_{n}]$  for any $n\geq  N+1$. In consequence, $\text{mdim}_{\text{M}}(\MM,f)$ and $\text{mdim}_{\text{H}}(\MM,f)$ are not continuous on $d$.
\end{example}

In Example \ref{exampledesc}, we proved that there exists a dynamical system with metric mean dimension and mean Hausdorff  dimension maps not continuous with respect to the metrics. In the following theorem, we will prove that this result is more general.


 \begin{theorem}\label{dwfwfwfw}  Set  $\emph{Q}=\emph{M}$ or $\emph{H}$. If there exists a continuous map $f:\MM\rightarrow \MM$ such that $\emph{mdim}_{\emph{Q}}(\MM,d,f)>0$, for some $d\in \MM(\tau )$, then 
\begin{equation*}
\begin{aligned}
\emph{mdim}_{\emph{Q}}(\MM,f) \colon \MM(\tau) &\to \mathbb{R} \cup \lbrace \infty \rbrace\\
d &\mapsto \emph{mdim}_{\emph{Q}}(\MM,d,f)
\end{aligned}
\end{equation*}
is not continuous anywhere. 
\end{theorem}
\begin{proof}   Let $(\MM,d)$ be a compact metric space and $f:\MM\to \MM$ be a continuous map such that $\text{mdim}_{\text{M}}(\MM,d,f)>0$. Given any $\alpha,\varepsilon\in (0,1)$, we define the metric
$$d_{\alpha,\varepsilon}(x,y)=\left\lbrace\begin{array}{cccc}
d(x,y), \text{ if } d(x,y)\geq \varepsilon,\\
\varepsilon^{1-\alpha}d(x,y)^\alpha, \text{ if } d(x,y)<\varepsilon.
\end{array}\right.$$
Note that   $d_{\alpha,\varepsilon}\in \MM(\tau)$. Moreover, taking $x,y\in \MM$ such that $d(x,y)\geq \varepsilon$, we have that $|d(x,y)-d_{\alpha,\varepsilon}(x,y)|=0<\varepsilon$. On the other hand, if we consider $x,y\in \MM$ such that $d(x,y)<\varepsilon$, we have that
\begin{eqnarray*}
|d(x,y)-d_{\alpha,\varepsilon}(x,y)|=|d(x,y)-\varepsilon^{1-\alpha} d(x,y)^\alpha |
\leq d(x,y) + \varepsilon^{1-\alpha} d(x,y)^\alpha < 2\varepsilon.
\end{eqnarray*}
Hence, $D(d,d_{\alpha,\varepsilon})<2\varepsilon$. However, for $\text{Q}=\text{M}$ or $\text{H}$ we prove
$$\text{mdim}_{\text{Q}}(\MM,d_{\alpha,\varepsilon},f)=\frac{\text{mdim}_{\text{Q}}(\MM,d,f)}{\alpha}.$$

Firstly, we prove to claim for metric mean dimension. Consider any $\eta\in(0,\varepsilon)$. Let $A$ an $(n,f,\eta)$-spanning set of $(\MM,d)$. Then, for any $y\in \MM$, there exists $x\in A$ such that $d_n(x,y)<\eta$. Hence,
 $$(d_{\alpha,\varepsilon})_{n}(x,y)=\varepsilon^{1-\alpha}d_n(x,y)^\alpha<\varepsilon^{1-\alpha}\eta^\alpha.$$ 
 Thus, $A$ is an $(n,f,\varepsilon^{1-\alpha}\eta^\alpha)$-spanning set of $(\MM,d_{\alpha,\varepsilon})$. Therefore, 
 $$\text{span}_{d_{\alpha,\varepsilon}}(f,\varepsilon^{1-\alpha}\eta^\alpha)\leq \text{span}_{d}(f,\eta),$$
and consequently, we obtain that
\begin{eqnarray}\label{1a}
\text{mdim}_{\text{M}}(\MM ,d_{\alpha,\varepsilon},f)=\displaystyle\lim_{\eta\to 0} \frac{\text{span}_{d_{\alpha,\varepsilon}}(f,\varepsilon^{1-\alpha}\eta^\alpha)}{|\log(\varepsilon^{1-\alpha} \eta^\alpha) |} &\leq & \displaystyle\lim_{\eta\to 0} \frac{\text{span}_{d}(f,\eta)}{\alpha|\log \eta |} \frac{|\log (\eta^\alpha)|}{|\log (\varepsilon^{1-\alpha}\eta^\alpha)|}\nonumber\\
&=& \frac{\text{mdim}_{\text{M}}(\MM,
d,f)}{\alpha}.\end{eqnarray}

On the other hand, notice that, for any $x,y\in \MM$ such that   $(d_{\alpha,\varepsilon})_{n}(x,y)<\varepsilon$, we have that $d_{n}(x,y)<\varepsilon$, because otherwise  $ (d_{\alpha,\varepsilon})_{n}(x,y)=d_{n}(x,y)\geq \varepsilon$. Let $E$ be an $(n,f,\eta)$-spanning set of $(\MM,d_{\alpha,\varepsilon})$, where $\eta\in (0,\varepsilon)$.  Then, for any $y\in \MM$, there exists $x\in E$ with $(d_{\alpha,\varepsilon})_n(x,y)<\eta$ and it follows that
$$(d_{\alpha,\varepsilon})_{n}(x,y)=\varepsilon^{1-\alpha}d_{n}(x,y)^\alpha<\eta<\varepsilon \Rightarrow  d_n(x,y)<\varepsilon^{\frac{\alpha-1}{\alpha}}\eta^{\frac{1}{\alpha}}.$$
Thus, $E$ is an $(n,f,\varepsilon^{\frac{\alpha-1}{\alpha}}\eta^{\frac{1}{\alpha}})$-spanning set of $(\MM,d)$ and therefore
$$\text{span}_{d_{\alpha,\varepsilon}}(f,\eta)\geq \text{span}_{d}(f,\varepsilon^{\frac{\alpha-1}{\alpha}}\eta^{\frac{1}{\alpha}}).$$
Hence, 
\begin{eqnarray}\label{1ab}
\text{mdim}_{\text{M}}(\MM,f,d_{\alpha,\varepsilon})=\displaystyle\lim_{\eta\to 0} \frac{\text{span}_{d_{\alpha,\varepsilon}}(f,\eta)}{|\log (\eta) |}&\geq & \displaystyle\lim_{\eta\to 0} \frac{\text{span}_{d}(f,\varepsilon^{\frac{\alpha-1}{\alpha}}\eta^{\frac{1}{\alpha}})}{|\log (\varepsilon^{\frac{\alpha-1}{\alpha}}\eta^{\frac{1}{\alpha}}) |}\frac{|\log (\varepsilon^{\frac{\alpha-1}{\alpha}}\eta^{\frac{1}{\alpha}}) |}{|\log \eta|}\nonumber\\
&=&\displaystyle\lim_{\eta\to 0} \frac{\text{span}_{d}(f,\varepsilon^{\frac{\alpha-1}{\alpha}}\eta^{\frac{1}{\alpha}})}{|\log (\varepsilon^{\frac{\alpha-1}{\alpha}}\eta^{\frac{1}{\alpha}}) |}\frac{|\log (\eta^{\frac{1}{\alpha}}) |}{|\log \eta|}\nonumber\\
&=& \frac{\text{mdim}_{\text{M}}(\MM,f,d)}{\alpha}.
\end{eqnarray}
It follows from \eqref{1a} and \eqref{1ab} that $\text{mdim}_{\text{M}}(\MM,f,d_{\alpha,\varepsilon})=\frac{\text{mdim}_{\text{M}}(\MM,f,d)}{\alpha}. $

Next, we prove the theorem for mean Hausdorff dimension. We will need the relation $$\text{mdim}_{\text{H}}(\MM,f,d^{\alpha})=\frac{\text{mdim}_{\text{H}}(\MM,f,d)}{\alpha},\quad\text{ for any }\alpha\in(0,1),$$ which will be shown in Example \ref{metrica}. Fix $\eta\in (0,\varepsilon)$. For every $x,y\in \MM$ with $d_{n}(x,y)<\eta,$  we have that $(d_{\alpha,\epsilon})_{n}(x,y)=\varepsilon^{1-\alpha}d_n(x,y)^\alpha$. Thus, for all $E\subset M$ such that $\text{diam}_{d_n^\alpha}(E)<\eta$, we have that $\text{diam}_{(d_{\alpha,\varepsilon})_n}(E)<\varepsilon^{1-\alpha}\eta$. Therefore $$H_{\varepsilon^{1-\alpha}\eta}^s(\MM,(d_{\alpha,\varepsilon})_{n})\leq H_\eta^s(\MM,d_n^\alpha), \quad\text{for every }0<\eta<\epsilon.$$ Thus,
\begin{equation}\label{ncdmfezx}\text{mdim}_{\text{H}}(\MM,d_{\alpha,\varepsilon},f)\leq\text{mdim}_{\text{H}}(\MM,d^\alpha , f) = \frac{\text{mdim}_{\text{H}}(\MM,d,f)}{\alpha}.\end{equation}  

On the other hand, given $\eta\in (0,\varepsilon)$, we have for every $x,y\in \MM$, with  $d_{n}(x,y)<\eta$, that  
$$(d_{\alpha,\epsilon})_{n}(x,y)= \varepsilon^{1-\alpha} d_n(x,y)^\alpha> \eta^{1-\alpha} d_n(x,y)^\alpha.
 $$ Thus, for all $E\subset\MM$ with $\text{diam}_{(d_{\alpha,\varepsilon})_{n}}(E)<\eta$, it follows that $\text{diam}_{d_n^\alpha}(E)<\eta^\alpha$. Therefore, we obtain that $$H_\eta^s(\MM,(d_{\alpha,\varepsilon})_{n})\geq H_{\eta^\alpha}^s(\MM,d_n^\alpha).$$ Consequently,
\begin{equation}\label{zxxxm}\text{mdim}_{\text{H}}(\MM ,d_{\alpha,\varepsilon},f)\geq\text{mdim}_{\text{H}}(\MM ,d^\alpha,f) = \frac{\text{mdim}_{\text{H}}(\MM,d,f)}{\alpha}.\end{equation}
It follows from \eqref{ncdmfezx} and \eqref{zxxxm} that $\text{mdim}_{\text{H}}(\MM ,d_{\alpha,\varepsilon},f)=\frac{\text{mdim}_{\text{H}}(\MM,d,f)}{\alpha}.$

Next, given that 
$$\text{mdim}_{\text{M}}(\MM ,d_{\alpha,\varepsilon}, f)= \frac{\text{mdim}_{\text{M}}(\MM,d, f)}{\alpha}\quad \text{ and }\quad \text{mdim}_{\text{H}}(\MM,d_{\alpha,\varepsilon},f)= \frac{\text{mdim}_{\text{H}}(\MM,d,f)}{\alpha},$$ 
and $D(d_{\alpha,\varepsilon},d)<2\varepsilon$, for any $\varepsilon>0$, we can conclude that $\text{mdim}_{\text{M}}(\MM,d,f)$ and $\text{mdim}_{\text{H}}(\MM,d,f)$ are not continuous with respect to the metric.\end{proof}

\section{Composing metrics with subadditive continuous maps} \label{section7}
In this section, we will consider metrics in the set  
\begin{equation*}\label{SAM}
\mathcal{A}_{d}(\MM) = \lbrace g_d \colon g_d(x,y) = g (d(x,y))  \text{ for all }x, y\in  \MM,  \text{ and } g \in \mathcal{A}[0,\rho]  \rbrace,
\end{equation*}
where $\rho$ is the diameter of $\MM$ and \begin{equation*}\label{SA}
\mathcal{A}[0, \rho] = \left\{ g :[0,\rho]\rightarrow [0,\infty) : g\text{ is continuous,  increasing, subadditive  and }g^{-1}( 0) =\{ 0\}   \right\}.
\end{equation*}

Remember that $g:[0,\infty)\rightarrow [0,\infty)$ is called \textbf{subadditive} if $g(x + y) \leq g(x) + g(y)$ for all $x,y  .$  For instance, if $g$ is \textbf{concave} (that is, if $g(tx+(1-t)y)\geq tg(x)+(1-t)g(y)$, for any $t\in[0,1]$ and $x,y\in [0,\rho]$) and $g(0)\geq 0$,  then $g$ is subadditive. In fact, if $g:[0,\infty)\rightarrow [0,\infty)$ is concave and $g(0)=0$, then  $tg(x)\leq g(tx)$ for any $t\in[0,1]$ and $x\in [0,\infty)$. Hence, for any $x,y\in[0,\infty)$, taking  $t=\frac{x}{x+y}\in [0,1]$, we have $$ 
g(x)=g(t(x+y))\geq tg(x+y)\quad\text{and}\quad g(y)= g((1-t)(x+y))\geq (1-t)g(x+y).
$$ Therefore, $g(x)+g(y)\geq g(x+y).$

\begin{lemma}\label{metricg}
For any $g \in\mathcal{A}[0,\rho]$, we have  that:   
\begin{itemize}  \item[i)] $g_{d}$ is a metric on $\MM$. 
\item[ii)] $g_d \in \MM(\tau)$. Consequently, $\mathcal{A}_{d}(\MM) \subseteq  \MM(\tau)$.
\item[iii)] If $f:\MM\rightarrow\MM$ is a continuous map, then, for any $n\in\mathbb{N}$ and $x,y\in \MM$, we have $({g_{d}})_{n}(x,y)= g(d_{n}(x,y))$.
\end{itemize}
\end{lemma}

\begin{proof}   i) Clearly $g_d(x,y)\geq 0$ and $g_d(x,y) = g_d(y,x)$ hold. Furthermore,  since $g^{-1}\lbrace 0 \rbrace = \lbrace 0 \rbrace$, we have
$$g_d(x,y) = 0 \Leftrightarrow g(d(x,y)) = 0 \Leftrightarrow d(x,y) = 0 \Leftrightarrow x=y.$$
Next, since $g$ is increasing, then, for $x, y, z \in \MM$, it follows that
\begin{align*}
g_d(x,z) & = g(d(x,z)) \leq g(d(x,y) + d(y,z))  \leq g(d(x,y)) + g(d(y,z))
 = g_d(x,z) + g_d(z,y).
\end{align*}
Hence, $g_{d}$ is a metric on $\MM$. 
 
 \medskip

ii) We prove that, given any $x \in \MM$, then for any $\varepsilon > 0$  there is $\delta > 0$ such that
$B_{d}(x,\delta) \subset B_{g_d}(x,\varepsilon),$ where $B_{d^{\prime}}(x,\epsilon)$ denotes the open ball with center $x$ and radius $\epsilon>0$ with respect a metric $d^{\prime}$.
Indeed, since $g$ is  continuous at 0 and $g^{-1}\lbrace 0 \rbrace = \lbrace 0 \rbrace$, for all $\varepsilon > 0$, there is $\delta > 0$ such that if $0 \leq a < \delta$, then   $  0 \leq g(a) < \varepsilon.$  Thus,  for any $y \in \MM$ such that  $d(x,y) < \delta $, we have $g(d(x,y)) < \varepsilon$, that is, $  g_d(x,y) < \varepsilon.$ Therefore, $B_{d}(x,\delta) \subset B_{g_d}(x,\varepsilon).$
 
Next, we prove for all $x \in \MM$ and each $\varepsilon > 0$, there is $\delta> 0$ such that
  $B_{g_d}(x,\delta) \subset B_{d}(x,\varepsilon).$ We show that if $a,b \geq 0$ and $g(b) < \frac{g(a)}{2} $, then    $  b < \frac{a}{2}.$ Indeed, if $a\leq 2b$,  since $g$ is increasing and subadditive, then  we have 
$$g(a) \leq g(2b)  \leq 2g(b).$$  From this fact, setting   $\delta = \frac{g(\varepsilon)}{2}$, if $g_{d}(x,y)<\delta$,   we have
$$  g(d(x,y)) < \frac{g(\varepsilon)}{2}  \quad \Rightarrow \quad d(x,y) < \frac{\varepsilon}{2} < \varepsilon.$$ 
Therefore  $B_{g_d}(x,\delta) \subset B_{d}(x,\varepsilon).$ If follows from the   above facts that  $g_d \in \MM(\tau)$.

\medskip

\noindent iii) Fix a continuous map $f:\MM\rightarrow \MM$. Since $g$ is increasing, we have that $$g(d(f^{m}(x),f^{m}(y)))=\max \lbrace g(d(x,y)),g(d(f(x),f(y)))\dots, g(d(f^{n-1}(x),f^{n-1}(y))) \rbrace$$ if and only if $$d(f^{m}(x),f^{m}(y))=\max \lbrace d(x,y),d(f(x),f(y))\dots, d(f^{n-1}(x),f^{n-1}(y)) \rbrace .$$
Hence, given $n \in \NN$, we have for any $x,y\in\MM$ that 
\begin{align*}
(g_d)_{n}(x,y) & = \max \lbrace g_d(x,y),g_d(f(x),f(y))\dots, g_d(f^{n-1}(x),f^{n-1}(y)) \rbrace\\
&= \max \lbrace g(d(x,y)),g(d(f(x),f(y)))\dots, g(d(f^{n-1}(x),f^{n-1}(y))) \rbrace\\
&  =g\left( \max \lbrace d(x,y),d(f(x),f(y))\dots,  d(f^{n-1}(x),f^{n-1}(y)) \rbrace\right)  = g({d}_n(x,y)),
\end{align*}
which proves iii).
\end{proof} 
Next, we will consider the metric mean dimension with metrics on  $\mathcal{A}_{d}(\MM)$. For any continuous map $g\in \mathcal{A}[0, \rho]$, we will take
$$k_m(g) = \liminf_{\varepsilon \to 0^+} \frac{\log(g(\varepsilon))}{\log(\varepsilon)} \quad\text{ and }\quad k_M(g) = \limsup_{\varepsilon \to 0^+} \frac{\log(g(\varepsilon))}{\log(\varepsilon)}.$$ 

\begin{lemma}\label{wdwwfwwcbb}     For any $g\in \mathcal{A}[0, \rho]$, we have that  $k_m(g) \leq k_M(g)\leq 1$. \end{lemma}

 \begin{proof} Without loss of generality, we can assume that $\rho\in (0,1)$. We prove that there exists $m\in (0,\infty)$ such that $mx\leq g(x)$ for any $x\in [0,\rho]. $ Since $g$ is subadditive, we have that $$g(\rho)\leq 2g\left(\frac{\rho}{2}\right)\leq \cdots \leq 2^{n}g\left(\frac{\rho}{2^{n}}\right)\Rightarrow \frac{g(\rho)}{\rho}\leq \frac{g\left(\frac{\rho}{2}\right)}{\frac{\rho}{2}}\leq \cdots \leq \frac{g\left(\frac{\rho}{2^{n}}\right)}{\frac{\rho}{2^{n}}},$$ for any $n\in \mathbb{N}.$ 
If $0< y\leq \rho$, there exists $n\geq 0$ such that $ \frac{\rho}{2^{n+1}}\leq y\leq \frac{\rho}{2^{n}}$, and hence $\frac {2^{n}}{\rho}\leq \frac{1}{y}\leq \frac{2^{n+1}}{\rho} $. Thus, 
$$  \frac{g(\rho)}{\rho}\leq \frac{g\left(\frac{\rho}{2^{n+1}}\right)}{\frac{\rho}{2^{n+1}}} \leq \frac{g(y)}{\frac{\rho}{2^{n+1}}}=2\frac{g(y)}{\frac{\rho}{2^{n}}} \leq 2 \frac{g(y)}{y}.$$ Therefore, taking $m=\frac{g(\rho)}{2\rho}$, we have that $  my\leq g(y)$ for any $y\in [0,\rho]. $   Thus, for any $x\in(0,\rho]$, we have that $$\log mx\leq  \log g(x) \Rightarrow   - \log g(x)\leq - \log mx\Rightarrow   \frac{\log g(x)}{\log x}\leq \frac{\log mx}{\log x}.$$ Given that $\frac{\log mx}{\log x}\rightarrow 1$, as $x\rightarrow 0$, we have that  $k_m(g)\leq k_M(g)\leq 1$. 
\end{proof}

 From now on, we will suppose that $k_{m}(g),k_{M}(g)>0$. For instance, if there exists  $n\in \mathbb{N}$ and $\delta \in(0,1)$ such that with $  g(x)\leq  x^{\frac{1}{n}},$ for any $x\in (0,\delta]$, we have that $$   \log g(x)\leq \frac{1}{n} \log x \Rightarrow -\frac{1}{n}\log x\leq - \log g(x) \Rightarrow \frac{1}{n}\leq \frac{\log g(x)}{\log x}.$$ 
 
 We remark that there exists maps $g\in \mathcal{A}[0, \rho]$ such that $k_{m}(g)=k_{M}(g)=0$.   Indeed, if $g$ is defined as $g(x)=\frac{1}{\sqrt{{\log (\frac{1}{x})}}}$ for $x>0$ and $g(0)=0$, we can prove that $k_{M}(g)=0$ ($g(x)$ is the inverse map of the function $f:[0,\infty)\mapsto \mathbb{R}$ defined as  $f(x)=e^{-\frac{1}{x^{2}}}$ for $x>0$ and $f(0)=0$).

 \medskip

Remember that for any two sequences of non-negative real numbers $(a_{n})_{n\in\mathbb{N}}$ and $(b_{n})_{n\in\mathbb{N}}$, we always have: 
\begin{align}
   \limsup_{n\rightarrow \infty} a_{n}b_{n}\leq  \limsup_{n\rightarrow \infty} a_{n} \limsup_{n\rightarrow \infty} b_{n} \label{ffwfwfw}\\
  \liminf_{n\rightarrow \infty} a_{n}b_{n}\geq  \liminf_{n\rightarrow \infty} a_{n} \liminf_{n\rightarrow \infty} b_{n}\label{ffeewfwfw},  \end{align}
 whenever the right-hand side is not of the form $0\cdot \infty$. The equalities hold if $\underset{n\rightarrow \infty}{\lim} a_{n}$ exists. 
 These facts will be useful for the next proposition.

\begin{proposition}\label{propg}
Take $g\in \mathcal{A} [0, \rho] $, such that $k_{m}(g),k_{M}(g)>0$. Set  $g_d(x,y) = g \circ d(x, y)$ for all $x,y \in \MM$.  For any continuous map $f:\MM\rightarrow \MM,$ we have
\begin{itemize}
\item[i)] $ \underline{\emph{mdim}}_{\emph{M}}(\MM,d,f) \geq k_m(g)\underline{\emph{mdim}}_{\emph{M}}(\MM,g_d,f).$
\item[ii)] $ \overline{\emph{mdim}}_{\emph{M}}(\MM,d,f) \leq k_M(g)\overline{\emph{mdim}}_{\emph{M}}(\MM,g_d,f).$
\end{itemize}
\end{proposition}

\begin{proof} Given that $k_{m}(g),k_{M}(g)\in (0,1]$, we can use the properties given in \eqref{ffwfwfw} and \eqref{ffeewfwfw}.

i) Fix  $\varepsilon > 0$. If $d_n(x, y) < \varepsilon$, then   $(g_{d})_{n}(x, y)=g(d_n(x, y)) \leq g(\varepsilon)$,   because $g$ is increasing.    Thus, any $(n, f, \varepsilon)$-spanning subset with respect to $d$ is an $(n, f, g(\varepsilon))$-spanning subset with respect to $g_d$. Hence,
\begin{equation}\label{mfef}\text{span}_{{d}}(n,f,\varepsilon) \geq\text{span}_{g_d}(n,f,g(\varepsilon)).\end{equation} Furthermore, since $g$ is continuous and $g(0) = 0$, we have $\underset{\varepsilon \to 0}{\lim} g(\varepsilon)  =0$. Therefore,
\begin{align*}
\underline{\text{mdim}}_{\text{M}}(\MM,{d},f) & = \liminf_{\varepsilon \to 0} \limsup_{n \to \infty}   \frac{\log\text{span}_{{d}}(n,f,\varepsilon)}{n|\log(\varepsilon)|}\\
& = \liminf_{\varepsilon \to 0} \limsup_{n \to \infty}  \frac{\log\text{span}_{{d}}(n,f,\varepsilon)}{n|\log(\varepsilon)|} \frac{|\log(g(\varepsilon))|}{|\log(g(\varepsilon))|}\\
(\text{from } \eqref{mfef})\quad & \geq \liminf_{\varepsilon \to 0} \limsup_{n \to \infty}  \frac{\log\text{span}_{g_d}(n,f,g(\varepsilon))}{n|\log(g(\varepsilon))|} \frac{|\log(g(\varepsilon))|}{|\log(\varepsilon)|}\\
(\text{from } \eqref{ffeewfwfw})\quad & \geq k_{m}(g)  \liminf_{\varepsilon \to 0} \limsup_{n \to \infty}  \frac{\log\text{span}_{g_d}(n,f,g(\varepsilon))}{n|\log(g(\varepsilon))|}\\
& = k_{m}(g)  \underline{\text{mdim}}_{\text{M}}(\MM,g_d,f).
\end{align*}

ii)   Fix $n \in \NN$ and $\varepsilon > 0$. Let $A$ be an $(n,f,\varepsilon)$-separated subset with respect to $d$. Hence, for any $x,y \in A$ with $x \neq y$, we have
$
d_{n}(x,y) =  \underset{0 \leq j < n}{\max} \left\{ d(f^{j}(x),f^{j}(y)) \right\} > \varepsilon,
$
and, therefore, there exists $j_0 \in \lbrace 0, \dots, n-1 \rbrace$ such that $ 
d(f^{j_0}(x), f^{j_0}(y)) > \varepsilon.$ Since $g$ is increasing, it follows that
$
g\left( d(f^{j_0}(x), f^{j_0}(y))\right) \geq g(\varepsilon)$. Therefore,
$$(g_{d})_{n}(x,y) = \max_{0 \leq j < n}\left\{ g \left(d(f^{j}(x),f^{j}(y)) \right)\right\} \geq g(\varepsilon).$$
Hence, $A$ is an $(n, f, g(\varepsilon))$-separated subset with respect to $g_d$. Thus,
\begin{equation}\label{mbsgdhd}\text{sep}_{d}(n,f,\varepsilon) \leq \text{sep}_{g_d}(n,f,g(\varepsilon)).\end{equation}
Therefore,
\begin{align*}
\overline{\text{mdim}}_{\text{M}}(\MM,d,f) & = \limsup_{\varepsilon \to 0} \limsup_{n \to \infty}  \frac{\text{sep}_{{d}}(n,f,\varepsilon)}{n|\log(\varepsilon)|} = \limsup_{\varepsilon \to 0} \limsup_{n \to \infty}  \frac{\text{sep}_{{d}}(n,f,\varepsilon)}{n|\log(\varepsilon)|} \frac{|\log(g(\varepsilon))|}{|\log(g(\varepsilon))|}\\
(\text{from } \eqref{mbsgdhd})\quad & \leq \limsup_{\varepsilon \to 0} \limsup_{n \to \infty}  \frac{\text{sep}_{g_d}(n,f,g(\varepsilon))}{n|\log(g(\varepsilon))|} \frac{|\log(g(\varepsilon))|}{|\log(\varepsilon)|}\\
(\text{from } \eqref{ffwfwfw})\quad & \leq k_{M}(g)  \limsup_{\varepsilon \to 0} \limsup_{n \to \infty} \frac{\text{sep}_{g_d}(n,f,g(\varepsilon))}{n|\log(g(\varepsilon))|}  =k_{M} (g) \overline{\text{mdim}}_{\text{M}}(\MM,g_d,f).
\end{align*} Hence, $\overline{\text{mdim}}_{\text{M}}(\MM,d,f) \leq k_{M} (g) \overline{\text{mdim}}_{\text{M}}(\MM,g_d,f).$
\end{proof}

\begin{lemma} \label{new}
    For any $g \in \mathcal{A}[0, \rho]$ such that $k(g) = k_m(g) = k_M(g) >0$, we have that
$$\overline{\emph{mdim}}_{\emph{M}}(\MM,d,f) = k(g)\overline{\emph{mdim}}_{\emph{M}}(\MM,g_d,f)$$
and
$$\underline{\emph{mdim}}_{\emph{M}}(\MM,d,f) = k(g) \underline{\emph{mdim}}_{\emph{M}}(\MM,g_d,f).$$
\end{lemma}
 \begin{proof}
     From \eqref{mfef}, we have that \begin{align*}
\overline{\text{mdim}}_{\text{M}}(\MM,{d},f) & = \limsup_{\varepsilon \to 0} \limsup_{n \to \infty}   \frac{\log\text{span}_{{d}}(n,f,\varepsilon)}{n|\log(\varepsilon)|}\\
& \geq \limsup_{\varepsilon \to 0} \limsup_{n \to \infty}  \frac{\log\text{span}_{g_d}(n,f,g(\varepsilon))}{n|\log(g(\varepsilon))|} \frac{|\log(g(\varepsilon))|}{|\log(\varepsilon)|}\\
& = k (g)  \limsup_{\varepsilon \to 0} \limsup_{n \to \infty}  \frac{\log\text{span}_{g_d}(n,f,g(\varepsilon))}{n|\log(g(\varepsilon))|}\\
& = k (g)  \overline{\text{mdim}}_{\text{M}}(\MM,g_d,f).
\end{align*}
It follows from Proposition \ref{propg}, item ii, that $\overline{\text{mdim}}_{\text{M}}(\MM,{d},f)= k (g)  \overline{\text{mdim}}_{\text{M}}(\MM,g_d,f).$

Analogously, using \eqref{mbsgdhd} and Proposition \ref{propg}, item i, we can prove that $\underline{\text{mdim}}_{\text{M}}(\MM,d,f) = k(g) \underline{\text{mdim}}_{\text{M}}(\MM,g_d,f).$
 \end{proof}

From now on, we will assume that $\rho=\text{diam}_{d}(\MM)<1.$ Next, set $${\mathcal{A}}^{+}[0,\rho]:=\{g\in {\mathcal{A}}[0,\rho]: k_m(g)=k_M(g)>0\}.$$  We will choose a suitable topology for ${\mathcal{A}}^{+}[0,\rho]$. Fix $g \in {\mathcal{A}}^{+}[0,\rho]$. Since any $h\in {\mathcal{A}}^{+}[0,\rho]$ satisfies  $h(0)=0$, then we  must have   $d(g(x),h(x))\rightarrow 0,$ as $x\rightarrow 0$. For a fixed $\varepsilon >0$, set  \begin{equation}\label{efwwwqf} \tilde{B}(g,\varepsilon)=\left\{h\in {\mathcal{A}}^{+}[0,\rho]:   g(x)(x^{\varepsilon}-1)< h(x)-g(x) < g(x)\frac{(1-x^{\varepsilon})}{x^{\varepsilon}}, \text{ for  }x\in (0,\rho]\right\}.\end{equation}  Given that we are assuming that  $\rho<1$, notice that $g\in \tilde{B}(g,\varepsilon)$,   because $$g(x)(x^{\varepsilon}-1) <0< g(x)\frac{(1-x^{\varepsilon})}{x^{\varepsilon}} \quad\text{for  any  }x\in (0,\rho].$$   Furthermore,  if $h\in \tilde{B}(g,\varepsilon)$, then for any $x\in (0,\rho]$, we have that \begin{align*}g(x)(x^{\varepsilon}-1)&< h(x)-g(x)< g(x)\frac{(1-x^{\varepsilon})}{x^{\varepsilon}}\iff x^{\varepsilon}g(x)< h(x)< \frac{g(x)}{x^{\varepsilon}} \end{align*}
(see  Figure \ref{fig:ejemplo1d}). Let $\mathcal{T}$ be the topology induced by the sets $\tilde{B}(g,\varepsilon)$, that is, these sets form a subbase for $\mathcal{T}$. 
\begin{figure}[ht]
 \centering
   \includegraphics[width=0.26\textwidth]{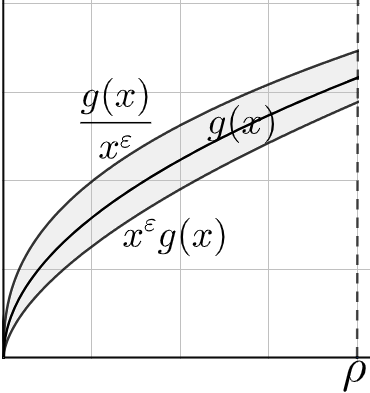}
\caption{$\tilde{B}(g,\varepsilon)$} \label{fig:ejemplo1d}
\end{figure}

\begin{lemma}\label{contikg} The map   
\begin{align*}
\mathcal{Z} : ( {\mathcal{A}}^{+}[0,\rho],\mathcal{T})    &\to   (0,1]  \\
g &\mapsto  k(g):=k_{m}(g)
\end{align*}
is continuous.  
\end{lemma}
\begin{proof}
     For any $g\in  {\mathcal{A}}^{+}[0,\rho]$, define 
    $$\tilde{g}(x)=\begin{cases}
       \frac{\log g(x)}{\log x},\text{ if } x\in (0,\rho]\\
       k(g),\text{ if } x=0.
    \end{cases}  $$
     Note that $\tilde{g}:[0,\rho]\rightarrow \mathbb{R}$ is a continuous map. Specifically,  $\tilde{g}$ is continuous at 0, because $$\tilde{g}(0)=k(g)=\lim_{x\rightarrow 0}\tilde{g}(x).$$ Next, fix  
     $h\in \tilde{B}(g,\varepsilon)$. Given that $\rho<1$, then for any $x\in (0,\rho]$ we have that \begin{align*}  x^{\varepsilon}g(x)< h(x)< \frac{g(x)}{x^{\varepsilon}} & \iff x^{\varepsilon}< \frac{h(x)}{g(x)}< \frac{1}{x^{\varepsilon}}\\
     &\iff \varepsilon\log x<\log h(x)-\log g(x)< -\varepsilon \log x.\end{align*}
Therefore,     $-\varepsilon< \tilde{g}(x)-\tilde{h}(x)<\varepsilon$ for any $x\in (0,\rho]$. Thus,  $|k(g)-k(h)|=|\tilde{g}(0)-\tilde{h}(0)|\leq\varepsilon$, by the continuity of both $\tilde{g}$ and $\tilde{h}$. This fact proves that $g\mapsto k(g)$ is a continuous map. 
\end{proof}

For the next results, we will consider the set $${\mathcal{A}}^{+}_d(\MM) =\{g\circ d \in \mathcal{A}_d(\MM): g\in {\mathcal{A}}^{+}[0,\rho]   \}.$$   Notice that ${\mathcal{A}}^{+}_d(\MM)\neq \emptyset$, because  the function  $g(x)=x^a$, for a fixed $a\in (0,1]$, belongs to $\mathcal{A}^{+}[0,\rho]$    (see Example \ref{metrica}). In particular, $d\in {\mathcal{A}}^{+}_d(\MM)$. 

\begin{lemma}
    Let $\MM$ be a compact space such that the  metric map $d:\MM\times \MM\rightarrow [0,\rho]$ is surjective. Then \begin{align*}
\mathcal{Z} : {\mathcal{A}}^{+}[0,\rho]     &\to  {\mathcal{A}}^{+}_{d}(\MM)   \\
g &\mapsto g\circ d
\end{align*} is a bijective map. 
\end{lemma}
\begin{proof} Clearly $\mathcal{Z} $ is surjective.     Next, we prove that for any $\tilde{d}\in {\mathcal{A}}^{+}_{d}(\MM)$, there exists a unique $g_{\tilde{d}}\in {\mathcal{A}}^{+}[0,\rho]$ such that $\tilde{d}=g\circ d$. Suppose that $g_{1},g_{2}\in {\mathcal{A}}^{+}[0,\rho]$ and $\tilde{d}=g_{1}\circ d =g_{2}\circ {d}$. Since $d$ is surjective, for any $t\in [0,\rho]$, there exist $x,y\in\MM$ such that $t=d(x,y)$. Therefore, $g_{1}(t)=g_{2}(t)$, as we want to prove. 
\end{proof}

 
 
 Suppose that $d:\MM\times \MM\rightarrow [0,\rho]$ is surjective. We will equip  ${\mathcal{A}}^{+}_d(\MM)$ with  the topology $\mathcal{W}$ which becomes the map  
\begin{align*}
\mathcal{Z} : ( {\mathcal{A}}^{+}[0,\rho],\mathcal{T})    &\to  ({\mathcal{A}}^{+}_{d}(\MM),\mathcal{W})  \\
g &\mapsto d
\end{align*}
  a homeomorphism. 

\begin{theorem}\label{teoee} Let $\MM$ be a compact space such that the  metric map $d:\MM\times \MM\rightarrow [0,\rho]$ is surjective. 
Suppose that $\emph{mdim}_{\emph{M}}(\MM,f,d)<\infty$.  The maps
\begin{align*}
\overline{\emph{mdim}}_{\emph{M}}(\MM,f) \colon ({\mathcal{A}}^{+}_{d}(\MM),\mathcal{W}) &\to \RR  \\
g_{d} &\mapsto \overline{\emph{mdim}}_{\emph{M}}(\MM,g_{d},f)
\end{align*} and \begin{align*}
\underline{\emph{mdim}}_{\emph{M}}(\MM,f) \colon ({\mathcal{A}}^{+}_{d}(\MM),\mathcal{W}) &\to \RR  \\
g_{d} &\mapsto \underline{\emph{mdim}}_{\emph{M}}(\MM,g_{d},f)
\end{align*}
are continuous.    
\end{theorem}
\begin{proof} We prove the case  $\overline{\text{mdim}}_{\text{M}}(\MM,f) \colon {\mathcal{A}}^{+}_{d}(\MM) \to \RR$, since the proof of the theorem is analogous for  the case $\underline{\text{mdim}}_{\text{M}}(\MM,f) \colon {\mathcal{A}}^{+}_{d}(\MM) \to \RR$.    If $\overline{\text{mdim}}_{\text{M}}(\MM,f,d)=0$, it follows from Lemma \ref{new} that $\overline{\text{mdim}}_{\text{M}}(\MM,f) \colon {\mathcal{A}}^{+}_{d}(\MM) \to \RR$ is the zero map.  

We will suppose that $0<\overline{\text{mdim}}_{\text{M}}(\MM,f,d)<\infty$.  Take $ \tilde{d}$   in  ${\mathcal{A}}^{+}_{d}(\MM)$ and let $g_{\tilde{d}}$  be the  unique map in ${\mathcal{A}}^{+}[0,\rho]$ such that $ \tilde{d}=g_{\tilde{d}}\circ d $. From Lemma \ref{new}, we have that $$  \overline{\text{mdim}}_{\text{M}}(\MM,f)(\tilde{d})=\overline{\text{mdim}}_{\text{M}}(\MM,f)(g_{\tilde{d}}\circ d)=\frac{\overline{\text{mdim}}_{\text{M}}(\MM,d,f)}{k(g_{\tilde{d}})} .$$
Hence, the continuity of $\overline{\text{mdim}}_\text{M}(\MM,f) \colon {\mathcal{A}}^{+}_{d}(\MM) \to \RR$ follows from Lemma \ref{contikg} and given that $k(g)>0$ for any $g\in{\mathcal{A}}^{+}[0,\rho]$.  
\end{proof}


\section{Additional examples}\label{sectionfinal}
 In this section we will present some examples of maps $g\in{\mathcal{A}}^{+}[0,\rho]$ and  the respective  expressions for $\text{mdim}_{\text{M}}(\MM ,g_d,f)$. 

\begin{example}\label{metrica}   Fix any $a \in (0, 1]$. Consider the function $g(x) = x^a$ defined for all $x \in [0,\infty)$. Notice  that $g(x+y) \leq g(x) + g(y)$ for any $x, y \geq 0$. Next, by defining $g_d(x, y) = d(x, y)^a$, we find that $k(g) = a$, and therefore
\begin{equation}\label{vdvevxzz}\text{mdim}_{\text{M}}(\MM ,g_d,f) = \frac{\text{mdim}_{\text{M}}(\MM,d,f)}{a}.\end{equation}
For instance, we have that 
\begin{equation}\label{fwfwfswfw}\text{mdim}_{\text{M}}(([0,1]^{n})^{\mathbb{Z}},h_{\textbf{d}},\sigma) = \frac{n}{a},\end{equation} where is the metric defined in Theorem \ref{theo1}  and $\sigma:([0,1]^{n})^{\mathbb{Z}}\rightarrow ([0,1]^{n})^{\mathbb{Z}}$ is the left shift.
\end{example} 

\begin{example}\label{metricados} Fix any $a \in (0, 1]$. Consider the function $g(x) = x^a$ defined for all $x \in [0,\infty)$. We will prove that 
\begin{equation}\label{gegwvxfmhm}\text{mdim}_{\text{H}}(\MM ,g_{d},f)=\frac{1}{a} \text{mdim}_{\text{H}}(\MM,d,f) .\end{equation}
In fact, consider a fixed $a\in(0,1]$. In fact, consider any $a\in(0,1]$ fixed. Given any $\eta>0$, we have that  $d(x,y)\leq \eta$ if and only if  $d(x,y)^a\leq \eta^a$. Hence, it follows that
\begin{eqnarray*}
\text{H}_{\eta^a}^{s} (\MM,(g_{d})_n)&=& \inf\left\{ \Sigma_{k=1}^{\infty}(\text{diam}_{d_{n}^a} (E_{k}))^{s}: \MM=\cup_{k=1}^{\infty} E_{k} \text{ with } \text{diam}_{d_{n}^a} (E_{k})<\eta^a\text{ for all }k\geq 1\right\}\\
&=  & \inf\left\{ \Sigma_{k=1}^{\infty}(\text{diam}_{d_{n}^a} (E_{k}))^{s}: \MM=\cup_{k=1}^{\infty} E_{k} \text{ with } \text{diam}_{d_{n}} (E_{k})<\eta\text{ for all }k\geq 1\right\}\\
&=&\inf\left\{ \Sigma_{k=1}^{\infty}(\text{diam}_{d_{n}} (E_{k}))^{as}: \MM=\cup_{k=1}^{\infty} E_{k} \text{ with } \text{diam}_{d_{n}} (E_{k})<\eta\text{ for all }k\geq 1\right\}\\
&=&\text{H}_{\eta}^{as} (\MM,d_n).
\end{eqnarray*}
Hence,
\begin{eqnarray*}
\dim_{\text{H}}(\MM,(g_{d})_n,\eta^a)&=&\sup\{s\geq 0:\text{H}_{\eta^a}^{s} (\MM,(g_{d})_n)\geq 1\}= \sup\{s\geq 0:\text{H}_{\eta}^{a s} (\MM,d_{n})\geq 1\}\\
&=&  \frac{1}{a}\sup\{a s\geq 0:\text{H}_{\eta}^{a s} (\MM,d_{n})\geq 1\} = \frac{1}{a} \dim_{\text{H}}(\MM,d_n,\eta),
\end{eqnarray*}
This fact proves \eqref{gegwvxfmhm}.  
\end{example}

 Let $f:\MM\rightarrow \MM$ be a continuous map such that  $\text{mdim}_{\text{M}}(\MM,d,f) > 0$.  It   follows from  Example \ref{metrica}   that  the image of the map
$ \text{mdim}_{\text{M}}(\MM,f) \colon {\mathcal{A}}^{+}_{d}(\MM) \to \RR \cup \lbrace \infty\}$ 
contains the interval $[\text{mdim}_{\text{M}}(\MM,d,f), \infty)$. Hence,
$$\sup_{d' \in \MM(\tau)} \text{mdim}_{\text{M}}(\MM,d' , f) = \infty.$$ Similar fact holds for the mean Hausdorff dimension.

\begin{example}\label{efjfk}
Consider $g(x) = \log(1 + x)$. Since   $1 + x + y \leq  1 + x + y + xy,$ we have 
\begin{align*}g(x + y) &= \log(1 + x + y) \leq \log ((1 + x)(1 + y))   = \log(1 + x) + \log (1 + y) = g(x) + g(y). \end{align*}
 Hence, $g $ is subadditive.  
Note that if $g_{1}$ and $  g_{2} \in {\mathcal{A}}^{+}[0, \infty)$, then $g_{1} \circ g_{2} \in {\mathcal{A}}^{+}[0,\infty)$. Consider $g_1(x) = x^a$, for $a \in (0, 1)$, and $g_2(x) = \log (1 + x)$. The composition $h(x)=g_2 \circ g_1(x) = \log (1 + x^a)$ belongs to ${\mathcal{A}}^{+}[0, \infty)$. We can prove that $k (h)= a$. Hence $$\text{mdim}_{\text{M}}(\MM,h_d,f) = \frac{\text{mdim}_{\text{M}}(\MM,d,f)}{a}.$$
\end{example}

\begin{example} Suppose that  $h:\MM \rightarrow \MM$ is $ \alpha$-H\"older for some $\alpha\in(0,1)$, that is,  there exists $K>0$ such that
\[d(h(x),h(y))\leq K d(x,y)^{\alpha}\quad\text{for all }x,y\in \MM.\]
 Setting  $ d_h(x,y) = d(h(x),h(y)) $   for all $x,y \in \MM$, we have respectively  from Examples \ref{metrica}  and \ref{metricados} that 
$$\text{mdim}_{\text{M}}(\MM, d_h,f) \leq \text{mdim}_{\text{M}}(\MM, d^{\alpha},f)=\frac{ \text{mdim}_{\text{M}}(\MM, d,f)}{\alpha}$$ and $$ \text{mdim}_{\text{H}}(\MM, d_h,f) \leq  \text{mdim}_{\text{H}}(\MM, d^{\alpha},f)=\frac{ \text{mdim}_{\text{H}}(\MM, d,f)}{\alpha}.$$  
\end{example}

If $\MM$ is a compact Riemannian manifold with $\text{dim}(\MM)\geq 2$, then the set $\mathcal{G}$ consisting of homeomorphisms  with positive metric mean dimension is residual in $\text{Hom}(\MM)$ (see \cite{Carvalho}). Therefore, for any $f\in \mathcal{G}$, we have  $$0=\text{mdim}(\MM , f)<\sup_{d' \in \MM(\tau)} \text{mdim}_{\text{M}}(\MM,d',f) = \sup_{d' \in \MM(\tau)} \text{dim}_{\text{B}}(\MM,d')= \infty,$$ where the first equality is because   $\MM$ is finite dimensional (see  \cite{lind}, page 6). Similar result holds for the case of mean Hausdorff dimension, following the facts proved in  \cite{Muentes3}.

\medskip 
 
 \begin{center}
{\bf Acknowledgements}
\end{center}

We would like to thank the anonymous Reviewers who contributed to the revision of this work. In special, we thank the Reviewer for the snowflake metric argument and for improving the readability of this text.

 \end{document}